\documentclass[11pt,a4paper,reqno,twoside]{amsart}
 \usepackage[margin=1in]{geometry}
 \usepackage{float}
 \usepackage{amsfonts}
 \usepackage{amsmath}
 \usepackage{graphics}
 \usepackage{epsfig}
 \usepackage{latexsym}
 \usepackage{times}
 \usepackage{amscd}
 \usepackage{ae,aecompl}
\usepackage{tikz,pgf,tikz-cd}
\usepackage{todonotes}
\usepackage{setspace}
\usepackage{fancyhdr}
\usepackage{lipsum}
\usepackage{amssymb}
\usepackage{amsthm}
 \usepackage{graphicx}
 \usepackage{ifthen}
 \usepackage{float}
 \usepackage{calrsfs}

\usepackage[pdfencoding=auto]{hyperref}

\newtheorem{thm}{Theorem}[section]
\numberwithin{equation}{section}
 \newtheorem{lem}[thm]{Lemma}        
 \newtheorem{prop}[thm]{Proposition}  
 \newtheorem{con}[thm]{Conjecture}  

 \theoremstyle{definition}

 \newtheorem{dfn}[thm]{Definition}
  
\newtheorem*{claim*}{Claim}

\newtheorem{rmk}[thm]{Remark}
\newtheorem{ex}[thm]{Example}

\allowdisplaybreaks[1]

\newcommand\set[1]{\left\{#1\right\}}

\newcommand{\cT}{\mathcal{T}}

\newcommand{\cH}{\mathcal{H}}
\newcommand{\cG}{\mathcal{G}}
\newcommand{\cA}{\mathcal{A}}
\newcommand{\cR}{\mathcal{R}}
\newcommand{\cC}{\mathcal{C}}
\newcommand{\cD}{\mathcal{D}}

\newcommand{\cI}{\mathcal{I}}

\newcommand{\cU}{\mathcal{U}}

\newcommand{\cB}{\mathcal{B}}

\newcommand{\cW}{\mathcal{W}}
\newcommand{\cO}{\mathcal{O}}
\newcommand{\cM}{\mathcal{M}}
\newcommand{\cN}{\mathcal{N}}

\newcommand{\CC}{\mathbb{C}}
\newcommand{\RR}{\mathbb{R}}

\newcommand{\ZZ}{\mathbb{Z}}

\newcommand{\TT}{\mathbb{T}}
\newcommand{\bN}{\mathbb{N}}

\DeclareMathOperator{\Fr}{Fr}

\DeclareMathOperator{\Sym}{Sym}

\newcommand{\bk}{\mathbf{k}}

\newcommand{\row}{\mathrm{r}}
\newcommand{\rd}{{\mathrm d}}

\newcommand{\cMo}{\cM^0}
\newcommand{\Mon}{\bar{\cM}}
\newcommand{\unicov}[1]{\tilde{\cM}^0_{#1}}
\newcommand{\cunicov}[1]{\widehat{\cM}^0_{#1}}
\DeclareMathOperator{\Rat}{Rat}
\newcommand{\MB}{\mathrm{M}}
\newcommand{\FB}{\mathrm{B}}
\newcommand{\bT}{{}^{\bo}T}

\newcommand{\fT}{{}^{\Phi}T}

\newcommand{\wt}{\widetilde}
\newcommand{\wh}{\widehat}
\newcommand{\ol}{\overline}

\DeclareMathOperator{\bo}{b}
\DeclareMathOperator{\scat}{sc}

\newcommand{\ve}{\varepsilon}

\DeclareMathOperator{\Stab}{Stab}

\renewcommand{\Im}{\operatorname{Im}}

\DeclareMathOperator{\SU}{SU}

\DeclareMathOperator{\dist}{dist}
\DeclareMathOperator{\Vect}{Vect}
\DeclareMathOperator{\tr}{tr}

\newcommand{\abs}[1]{\left| #1 \right|}

\newcommand{\p}{\partial}
\newcommand{\pa}{\p}

\renewcommand{\geq}{\geqslant}
\renewcommand{\leq}{\leqslant}
\renewcommand{\ge}{\geqslant}
\renewcommand{\le}{\leqslant}

\title[Monopoles and the Sen Conjecture]{Monopoles and the Sen Conjecture}

\author{Karsten Fritzsch}
\email{k.fritzsch@math.uni-hannover.de}
\address{Leibniz University Hannover}
\author{Chris Kottke}
\email{ckottke@ncf.edu}
\address{New College of Florida}
\author{Michael Singer}
\email{michael.singer@ucl.ac.uk}
\address{Department of Mathematics, University College, London WC1E 6BT}

\begin{document}
\maketitle

\begin{abstract}
We describe compactifications of the moduli spaces of SU(2) monopoles on $\RR^3$ as manifolds 
with corners, with respect to which the hyperK\"ahler metrics admit asymptotic expansions up to each boundary face. 
The boundary faces encode monopoles of charge $k$ decomposing into widely separated monopoles of lower charge, and the leading order
asymptotic of the metric generalizes the one obtained by Gibbons, Manton and Bielawski in the case of complete decomposition into monopoles of unit charge.
From the structure of the compactifications, we prove part of Sen's conjecture for the $L^2$ cohomology of the strongly centered moduli spaces by adapting
an argument of Segal and Selby.
\end{abstract}
\setcounter{tocdepth}{2}
\tableofcontents

\newpage
\part{A Description of the Compactification}

\section{Introduction}\label{intro}

The moduli space $\cN_k$ of non-abelian magnetic monopoles of charge
$k$ (gauge group $\mathrm{SU}(2)$) has received much attention
from both mathematicians and mathematical physicists.  It is well known that for
each positive integer $k$, $\cN_k$ is a complete, non-compact hyperK\"ahler
manifold of dimension $4k$.   The question of the asymptotic behaviour
of the metric $g_k$, say, on $\cN_k$ has also been studied in various
special cases
\cite{bielawski1995,RB_GM,bielawski2008,B_future,gibbons1995,AH,KS}.
Apart from the intrinsic interest in understanding the asymptotic
behaviour of this metric completely, it is also essential for the
study of the $L^2$ harmonic forms on the monopole moduli spaces, which
are the subject of the Sen Conjecture \cite{sen1994strong,SS}.

In this paper we shall give a complete description of the asymptotic
behaviour of $g_k$, the complete proof of which will appear in Part II,
and we shall combine our results with an argument
of Segal and Selby \cite{SS} to prove the `coprime case' of the Sen
Conjecture.  This argument is quite `soft' and exploits some very
specific features of the asymptotic geometry of the monopole moduli
spaces.  In particular, detailed analysis of the Hodge-de Rham
operator $\rd + \rd^*$ on $\cN_k$ is not required.  It would seem that
such an analysis will, however, be needed to prove the other cases of the Sen
Conjecture; we hope to return to this in the near future.

%
%

In order to state the Sen Conjecture, recall first the definition of the {\em strongly
  centred space} $\unicov{k}$ of monopoles of charge $k$.  This is the
universal cover of the quotient $\cN_k/\RR^3\times S^1$, where
$\RR^3$
acts by translations and $S^1$ by rotations of the framing (or `large'
gauge transformations).  The quotient has fundamental group $\ZZ_k$,
which we identify with the group of complex $k$-th roots of unity
\cite{AH,SS}.  Since $\RR^3\times S^1$ acts isometrically on $\cN_k$,
the quotient and its universal cover inherit a natural metric\footnote{In fact thinking of $\cM_k$ as the space of
  centred monopoles, $\cM_k/S^1$ can be identified as an $S^1$-hyperK\"ahler
  quotient of $\cN_k$. In particular $\cM_k/S^1$ inherits a
  hyperK\"ahler metric from $\cN_k$.}  from
that of $\cN_k$.
If $\zeta \in \ZZ_k$, denote by $\alpha_\zeta$ the
corresponding deck transformation, which will be an isometry of $\unicov{k}$.

The Sen Conjecture predicts the dimension of the space
$\cH^i(\unicov{k})$, the space of $L^2$ harmonic $i$-forms on
$\unicov{k}$.  More precisely, let
\begin{equation}
\cH^i_{k,\ell} =
\{ u \in \cH^i(\unicov{k}) : \alpha_\zeta^*u = \zeta^{\ell}u
\},\;\;\ell =0,1,\ldots,k-1.
\end{equation}
(The physics interpretation of $\ell$ is as the {\em electric charge}
of the quantum state $u$.)

\begin {con}[Sen Conjecture] \cite{sen1994strong,SS}
 \begin{enumerate}
  \item[(S.1)] If $k$ and $\ell$ are coprime, then
    $\cH^{2k-2}_{k,\ell} \cong \CC$, while  $\cH^i_{k,\ell} = 0$
    for $i \neq 2k-2$;
  \item[(S.2)] if $k$ and $\ell$ are not coprime, then $\cH^i_{k,\ell} = 0$ for all $i$.
 \end{enumerate}
\end{con}

We shall prove the `coprime case' (S.1) of the conjecture, along the
lines suggested in \cite{SS}.
\begin{thm}\label{sen.1a}
Statement \emph{(S.1)} of the Sen Conjecture holds true.
\end{thm}

\subsection{A Metric Compactification of \texorpdfstring{$\cM_k$}{M(k)}}\label{sec.intro.metr.cpct}

Denote by $\cM_k$ the quotient $\cN_k/\RR^3$ of the moduli space by
translations.  It is almost equivalent to think of $\cM_k$ as the
space of monopoles centred at the origin in $\RR^3$.  If $k\geq 2$,
$\cM_k$ is still non-compact and it is really the non-compactness of
the translation-group $\RR^3$ which underlies the non-compactness $\cM_k$ itself.  To explain this, let $m^{\nu}\in \cM_k$ be a
divergent sequence.  Then to paraphrase \cite[Proposition 3.8]{AH}, a
subsequence of $m^{\nu}$ consists of `widely separated monopoles of
type $a$' for some proper partition $a = (k_1,\ldots,k_n)$ of $k$.
({\em Proper} means
$n\geq 2$ and all $k_j\geq 1$).  The condition of
wide separation means that there
is a configuration of points $(p_1^{\nu},\ldots,p_n^{\nu})$ such that
\begin{equation}\label{e2.25.10.17}
\ve^{-1} :=\min_{i<j}|p_i - p_j| \gg 1
\end{equation}
and a collection of centred monopoles
\begin{equation}\label{e3.25.10.17}
(m_1,\ldots,m_n) \in \cM_{k_1} \times \cdots \cM_{k_n}
\end{equation}
such that for each $j$, $m^\nu(z-p^\nu_j)$ converges to
$m_j(z)$ on any fixed ball $\{|z|<R\}$.  In other words for large $\nu$,
$m$ looks like an approximate superposition of the translated monopoles $m_j(\cdot +
p^{\nu}_j)$.   (NB: given $a$, there is a subgroup of the symmetric
group $\Sigma_n$ which acts on configurations of type $a$,
consisting of those permutations $\sigma$ of $\{1,\ldots,n\}$ with $k_{\sigma(i)} = k_i$,
for all $i$.  A widely separated configuration of monopoles really
involves unordered configurations of points (and monopoles) where we
factor out by this group action. See \S\ref{sec.sym.mbs} for more detail.)

\begin{rmk}
This classification of divergent sequences according to `type'
strongly suggests that $\cM_k$ should have asymptotic regions which
correspond to the different types of divergent sequences in $\cM_k$.
This intuition is supported by intuition coming from the identifcation
of $\cN_k$ with $\Rat_k$, the space of based rational functions of
degree $k$ \cite{skd1984}.   The basic idea is that if $f_1$ and $f_2$
are rational functions respectively of degrees $k_1$ and $k_2$, then
generically their sum $f_1+f_2$ will be a rational function of degree
$k_1 +k_2$.  There are, however, subtleties in using this to try to
describe the asymptotic regions of $\cN_k$ because (for example) the identification
$\cN_k = \Rat_k$ breaks the symmetry of $\cN_k$ by singling out a
direction in $\RR^3$.  However, the idea is largely captured by the cover of $\cM_k$
by open sets corresponding to `decomposable monopoles' which appears
in \S\ref{S:cluster} of this paper.

We should also note that the case $a=(1,\ldots,1)$ is well understood
\cite{JT,RB_GM,gibbons1995,AH} and that $L^2$ harmonic forms on the Atiyah-Hitchin manifold $\cM_2/\TT$ have been studied in \cite[7.1.2]{HHM}; furthermore the results of
\cite{KS} give a partial description of such regions.
\end{rmk}

We shall introduce a compactification $\Mon_k$ of $\cM_k$, which will
be a \emph{manifold with corners} (MWC).  This provides a convenient
and powerful way to deal with the complexities of the asymptotic
geometry of $\cM_k$, including good definitions of the various
asymptotic regions, their intersections, and the behaviour of the
$L^2$ metric in each region.
A similar approach, using MWCs to
study complete Calabi--Yau metrics with complicated asymptotic
behaviour, can be found in \cite{qac}.  Manifolds with corners are
convenient for the study of many other non-compact and singular
problems in geometric analysis, see for example
\cite{gpaction, sigpack, fibcorn, MZ17, MZ18}.
Vasy's approach via MWCs to many-body geometry \cite{vasymb} underlies
our definition of `ideal configurations' of points in a euclidean
space and is an essential ingredient in our construction.

As a compact  manifold with corners, $\Mon_k$ has a finite number of boundary
hypersurfaces; these are indexed (at least for the moment) by proper
partitions $a$ of $k$ and denoted $N_a$.  To say that $\Mon_k$ is a
compactification of $\cM_k$ means that the interior of $\Mon_k$ is
$\cM_k$,
\begin{equation}
\Mon_k \setminus \bigcup_{a}N_a = \cM_k.
\end{equation}
Part of the definition of MWC is that the boundary hypersurfaces are
embedded. In particular, we may choose a boundary defining function
$\rho_a \geq 0$ for each $N_a$ and for sufficiently small
$\delta>0$, the sublevel set $U_a=\{\rho_a<\delta\}$ will be
diffeomorphic to the product $[0,\delta) \times N_a$.
The (hitherto ill-defined) asymptotic regions of $\cM_k$ can now be
defined precisely as the interiors of the $U_a$; these are
diffeomorphic to products $(0,\delta) \times (N_a)^\circ$, where of course
$(N_a)^\circ$ is the interior of the MWC $N_a$.

One of the advantages of $\Mon_k$ is that the corners structure
encodes the intersection properties of the different asymptotic
regions.  Recall that for partitions $a$ and $b$ of $k$, $a$ is {\em
  finer} than $b$, written $a\leq b$, if $b$ is obtained from $a$ by
bracketing terms in $a$. The boundary hypersurfaces $N_a$
and $N_b$ will intersect if and only if the corresponding partitions $a$ and $b$ of $k$
are {\em comparable}, that is $a\leq b$ or $b\leq a$.  For example, when $k=3$, we have the two
proper partitions $a = (1,1,1)$ and $b = (1,2)$ of $k$, and $a$ is a
refinement of (or simply finer than) $b$.   There are two asymptotic regions
of $\cM_3$ and their intersection consists of widely separated
monopoles of type $(1,1,1)$ with centres at $(p_1,p_2,p_3)$ such that
the distance $|p_1-p_2|$ is large
but much smaller than the distances $|p_1-p_3|$ and $|p_2-p_3|$.  Our
compactification handles these configurations through the
parameters
\begin{equation}\label{e1.30.7.18}
\rho_1 =\frac{1}{|p_1-p_2|},\; \rho_2 = \frac{|p_1-p_2|}{|p_1-p_3|}.
\end{equation}
which turn out to be local boundary defining functions for the two
boundary hypersurfaces.  To see why,
notice that if $\rho_1\to0$ with $\rho_2>0$ fixed, then $|p_1-p_2|$,
$|p_1-p_3|$, $|p_2-p_3|$ all tend to $\infty$ and the ratios between
them are all bounded; on the other hand, if $\rho_1>0$ is fixed and
$\rho_2\to 0$, then $|p_1-p_2|$ remains bounded while $|p_1-p_3|$ and
$|p_2-p_3|$ both go to $\infty$.  The two parameters
$(\rho_1,\rho_2)$ can thus be used to describe diverging triples
$(p_1^s,p_2^s,p_3^s)$ where $1\ll |p_1^s-p_2^s|\ll |p_1^s
- p^s_3|$.

More generally, $N_a \cap N_b$ will have a number of
disconnected components, corresponding to inequivalent ways of
bracketing the terms in $a$ to produce $b$.  The simplest example
occurs for $k=5$:
\begin{equation}\label{e1.29.8.18}
1+1+1+2 = (1+1+1)+2,\;\;1+1+1+2 = (1+1) + (1+2)
\end{equation}
both of which display the partition $1+1+1+2$ as a refinement
of $2+3$.   This is not a mere technicality as the two ways of
bracketing terms correspond to different intersections of asymptotic regions of
$\cM_5$. For this example, we have a $2$-monopole and a $3$-monopole,
widely separated. In the first case,
the $3$-monopole is in the $(1,1,1)$ asymptotic
region of $\cM_3$ while the $2$-monopole remains in a bounded subset
of $\cM_2$; in the second, the $3$-monopole is in the $(2,1)$
asymptotic region of $\cM_3$ and the $2$-monopole is in the asymptotic
$(1,1)$ region of $\cM_2$.

The situation is best described in terms of partitions $\lambda$
of the set $\bk = \{1,\ldots,k\}$. Such $\lambda$ has a type
$a=[\lambda]$, by taking the sizes of the blocks of
$\lambda$.  Equivalently, $\Sigma_k$ acts on the set of partitions
$\{\lambda\}$, and the set of orbits is precisely the set of partions
$\{a\}$ of $k$.
Refinement of partitions $\lambda$ (where $\lambda \leq \mu$ if every
block of $\mu$ is a union of blocks of $\lambda$) goes over to
refinement of integer partitions.  The point illustrated by the above
example is that the set of $\Sigma_k$-orbits of length-$2$ chains
$\lambda < \mu$ is {\em not} the same as the length-$2$ chains $a<b$
in the set of integer partitions: it is the former, not the latter,
that labels the codimension-$2$ hypersurfaces of $\Mon_k$.  More
generally, the codimension-$d$ corners of $\Mon_k$ are labelled by the
$\Sigma_k$-orbits of length-$d$ chains $\lambda_1 < \lambda_2 < \cdots
< \lambda_d$ in the set of partitions of $\bk$.

One should think of $N_a$ as the (compactified) moduli space of {\em ideal}
monopoles of type $a$---ideal in the sense of `infinitely
separated' configurations of points.  This will be made precise in the
next section: it is noteworthy that $N_a$ has a natural
definition as a manifold with corners, whereas to define `asymptotic
regions' requires arbitrary choices.

Let us now explain how asymptotic behaviour of the metric on $\cM_k$
is captured by the compactification $\Mon_k$.  The metric behaviour
reflects the additional structure of a fibration
\begin{equation}\label{e5.25.10.17}
\phi_a : N_a \to B_a,
\end{equation}
of each boundary hypersurface, where base and fibre are compact MWC.
The fibrations enjoy compatibility conditions at
the non-empty intersections $N_a \cap N_b$, giving
$\Mon_k$ an {\em iterated boundary fibration (IBF) structure}\footnote{Essentially the same structure
appears in \cite{gpaction,sigpack, fibcorn, qac} but there is unfortunately no
agreement on terminology} \cite{gpaction,sigpack,fibcorn,qac} which we
shall recall in \S\ref{sec.geometry} below.    Generally, if
$M$ is a compact MWC with an iterated boundary fibration structure, there is a smooth
vector bundle which we shall denote\footnote{Again, there is no
  agreement on terminology}
by $\fT M$, whose restriction to the interior $M^\circ$ is canonically isomorphic to
$TM^\circ$, but whose sections have particular decay properties at the
boundary (see \S\ref{sec.geometry}).

 A smooth metric on $\fT M$  (smooth up to and including all
boundary hypersurfaces) will automatically define a complete metric on
$M^\circ$ and the smoothness, as a metric on $\fT M$, captures precise
asymptotic behaviour near each boundary hypersurface.   A metric
arising in this way will be called a $\Phi$-metric.
Such metrics were first introduced in \cite{Anda-Rafe,qac}, where they are
referred to as `QAC' or `QFB' metrics.

Then our main theorem about the metric structure of $\cM_k$ is as
follows, a more precise version of which will be given in \S\ref{S:model_metrics}:
\begin{thm}\label{main.theorem}
The moduli space $\cM_k$ has a compactification $\Mon_k$ as a compact
MWC with iterated boundary fibration, and the $L^2$ metric $g_k$ extends to a smooth $\Phi$-metric on $\Mon_k$, which we denote by $g_k$ again.  Moreover, there is an
isometric $\TT$-action on $\Mon_k$ whose restriction to the interior
is the triholomorphic $\TT$-action on $\cM_k$, and whose orbits
are of bounded length with respect to $g_k$.
\end{thm}

Thus among the various possible compactifications of $\cM_k$
which arise from the different descriptions of the moduli space, our
compactification $\Mon_k$ is `metrically natural'.
This result, combined with a slight refinement of the argument in \cite[Sect. 3]{SS}, leads to a
quick proof of the coprime case of the Sen Conjecture.

The plan of the rest of this paper is as follows.  In Part I, we shall give our
main results about the moduli spaces, deferring the proof of the main
theorem, Theorem~\ref{main.theorem} to Part II.  We start in the next
section with compactifications of configuration spaces of points in a
euclidean space inspired by Vasy's resolved many-body spaces
\cite{vasymb}. In \S\ref{sec.geometry}, we recall the definition of
iterated boundary fibrations and then in \S\ref{sec.mon.cpctn} we
describe the compactification $\Mon_k$ and its
iterated boundary fibration structure.
 In \S \ref{S:cluster}, we explain, using the {\em a priori} estimates
 of Taubes, why any divergent sequence of monopoles in $\cM_k$ has a
 limit point in $\Mon_k$. Finally, in \S\ref{psen} we prove the
 coprime case of the Sen Conjecture, following an idea of
 Segal--Selby.   Sections \ref{sec.mon.cpctn} and \ref{psen} are
 independent of each other and can be read in either order.

\vspace{6pt}
{\noindent \bf Acknowledgements:} The authors are grateful to many
colleagues for useful conversations during the long gestation of this
project.  In particular, we 
thank Pierre Albin, Roger Bielawski, Daniel Grieser, Rafe Mazzeo,
Richard Melrose, Fr\'ed\'eric Rochon, Andy Royston and Andr\'as Vasy.

The research was supported in part by an EPSRC grant
EP/K036696/1.  The work was also supported by the NSF under Grant No.\
DMS-1440140 while the third author was in residence at the
Mathematical Sciences Research Institute in Berkeley, California,
during the Spring 2016 semester. The second author was supported by the NSF under Grant\ No.\ DMS-1811995.

\section{Geometric Preliminaries}\label{geom.prelim}

Before coming to the compactification of $\cM_k$, we devote this section to
the compactification of a simpler family of so-called `many-body spaces', in
particular the `reduced configuration spaces' $\RR^{3n}/\RR^3$ of $n \leq k$
points in $\RR^3$ up to translation. This interlude serves several purposes: first, these
spaces serve as simplified models for the monopole moduli spaces themselves
(indeed, these reduced configuration spaces, modulo action by the symmetric
group, are essentially equivalent to the moduli spaces of {\em abelian}
$\mathrm{U}(1)$ monopoles with appropriate framing); second, the appropriate
compactifications of many-body spaces are quite easy to construct, yet still illustrate
the essential combinatorial and geometric structure of our compactification $\Mon_k$ of
$\cM_k$, the construction of which is significantly more difficult; finally, the compactified
many-body space machinery plays a key technical role in our actual construction of $\Mon_k$ in Part II.
The forbearance of the reader is appreciated as we proceed to introduce a certain amount of notation.

We assume familiarity with the basic notions of manifolds with
corners as presented, for example in \cite{CCN} or \cite{daomwc}.
Other useful references are \cite{vasymb,gpaction,sigpack, fibcorn, qac}.


\subsection{Euclidean Many-Body Spaces and their Resolutions}\label{sec.mbs}


Let $V$ be a real euclidean vector space of dimension $N$. Let $\cW$ be
a finite family of linear subspaces of $V$ satisfying
\begin{eqnarray}
0,V &\in& \cW \label{e1.10.3.16} \\
W,W'\in \cW &\Rightarrow& W\cap W' \in \cW.\label{e2.10.3.16}
\end{eqnarray}
We refer to the second condition as `intersection closure' and to a such a family $\cW$ as a \emph{linear many-body structure}. The set
$\cW$ is partially ordered by inclusion.
Let $\ol{V}$ denote the radial compactification of $V$, and $\partial V$ its
boundary. It is best to think of $\partial V$
as the quotient $(V \setminus
0)/\RR^+$.  For $W\in\cW$ we denote by $\ol{W}$ and $\partial W$ the
corresponding radial compactifications and boundaries. Note that $\partial\{0\}
= \emptyset$.  It is also convenient to put
\begin{equation}\label{e2.31.7.18}
\partial \cW = \{ \p W : W \in \cW\}
\end{equation}
so that $\partial \cW$ is a set of submanifolds of $\partial V$.

\begin{dfn} If $\cW$ satisfies \eqref{e1.10.3.16} and \eqref{e2.10.3.16}, set
\begin{equation}\label{e6.10.3.16}
 \MB(V,\cW) := [\ol{V}; \partial\cW^*],
\end{equation}
and
\begin{equation}\label{e5.10.3.16}
 \FB(V,\cW) := [ \partial V; \partial \cW^* ]
\end{equation}
where $\cW^* := \cW\setminus \{V\}$. We call $\MB(V,\cW)$ the
\emph{many-body compactification} of $V$ with respect to $\cW$ and
$\FB(V,\cW)$ the \emph{free boundary} of $\MB(V,\cW)$.
\end{dfn}

\begin{rmk}
When it is clear what family $\cW$ is under consideration, and there
is no risk of confusion, we shall abbreviate the notation to
$\MB(V)$, and $\FB(V)$.
\end{rmk}

\begin{rmk}
We shall see below that the lift to $\MB(V,\cW)$ of $\p V$ is
$\FB(V,\cW)$, which is the reason for the terminology `free
boundary'.  This space is a natural compactification of the set of all
`ideal points' of $V$ which do not lie on any of the subspaces in $\cW^*$.
\end{rmk}

\begin{rmk} It is to be understood that the blow-ups are performed in
  size order.  The intersection-closure means that the blow-ups are
  well-defined; after $j$ blow-ups, the lifts of the remaining
  submanifolds are p-submanifolds, and those that intersect in the
  original family are disjoint when their intersection is blown up \cite{vasymb,kottke_mbcat}.
\end{rmk}

The important for example for us comes from the family of diagonals
in $E^k$, where $E$ is a (finite-dimensional) euclidean space.

\begin{ex}[Diagonals and configurations]\label{ex.diag.1}  Let $E =
  \RR^m$, and $V = E^k$ ($k$-fold product).  The set $\cD$ of
all diagonals of $E^k$ satisfies conditions
  \eqref{e1.10.3.16} and \eqref{e2.10.3.16} provided that we regard
  $0$ and $E^k$ itself as diagonals. In this case $\MB(E^k,\cD)$ is
  a natural compactification of the {\em configuration space of $k$
    points in $\RR^m$}.
\end{ex}

\begin{ex}[Reduced Configuration Spaces]\label{ex.diag.2} Continuing
  the previous example, note that all true diagonals (i.e.\ we exclude
  the subspace $0$) contain the minimal diagonal
$$
D_k = \{p_1= \cdots = p_k\}.
$$
Thus there is a quotient family
 \begin{equation}\label{E:reduced_mbsystem}
  \cD' = \{D/D_k : D \in \cD, D\neq 0\}
 \end{equation}
of linear subspaces of $E^k/E$ which is again intersection-closed. Then
$\MB(E^k/E,\cD')$ is the compactification of the space
of configurations of $k$ points mod translation, and
$\FB(E^k/E,\cD')$ is the space of ideal configurations
mod translation. We also refer to these spaces as (compactified)
{\em reduced configuration spaces}.
\end{ex}

If $V$ is a euclidean space with many-body structure $\cW$ and $A \in
\cW$, then both $A$ and $V/A$ inherit many-body structures.  The
many-body structure on $A$ is just the set of $W\in \cW$ with
$W\subset A$; that on $V/A$ is the set of quotients $W/A$, with
$W\supset A$.  We shall write $\MB(A)$, $\FB(A)$ for the many-body
compactification and free-boundary of $A$ with this induced many-body
structure and similarly for $V/A$.  Given the euclidean structure of
$V$ we can of course replace the quotient $V/A$ by $A^{\perp}$; then
the many-body structure is the set $\{W^{\perp} : W\in \cW, W\supset
A\}$.

These sub- and quotient-many-body structures appear when describing
the boundary faces of $\MB(V,\cW)$.
From the definition, the boundary hypersurfaces of $\MB(V,\cW)$ are
labelled precisely by the non-zero elements of $\cW$.  The free boundary
$\FB(V,\cW)$) corresponds to the element $V\in \cW$ and is the lift to
the blow-up of the boundary of $\ol{V}$.  Similarly,
the boundary hypersurfaces of
$\FB(V,\cW)$ are in one-one correspondence with the non-zero elements of
$\cW^*$.

\begin{thm}[\cite{kottke_mbcat}, Theorem~5.1]\label{thm.mbs.bhs} Let $0 \neq A\in \cW$. Then:
 \begin{enumerate}
  \item[(a)] The boundary hypersurface $N$ in $\MB(V,\cW)$ corresponding to $A$ is the compact MWC
\begin{equation}\label{e9.10.3.16}
N = \MB(V/A) \times \FB(A).
\end{equation}
  \item[(b)]  The boundary hypersurfaces $N_1$ and $N_2$ corresponding to
    $A_1,A_2\in \cW$ meet in $\MB(V)$ if and only if $A_1$ and $A_2$
    are comparable ($A_1\subset A_2$ or vice versa).  More generally
    the non-empty $d$-fold intersections of boundary hypersurfaces of
    $\MB(V)$ correspond precisely to length-$d$ chains
\begin{equation}\label{e2.29.8.18}
A_d\subset \cdots \subset A_2 \subset A_1
\end{equation}
of elements of $\cW$; every codimension-$d$ boundary face is a
connected component of such an intersection.  With $N_j$ the boundary
hypersurface corresponding to $A_j$, we have
\begin{equation}\label{e3.29.8.18}
N_1 \cap N_2\cap\cdots \cap N_d =
\MB(V/A_1) \times \FB(A_1/A_2) \times \cdots \times \FB(A_d).
\end{equation}
\end{enumerate}
\end{thm}

\begin{rmk} Here of course the many-body structure on $A_j/A_{j+1}$ is
  understood to be the family of subspaces $\{W/A_{j+1}: A_{j+1} \subset W \subset
  A_{j}\}$.
\end{rmk}

We refer to  \cite{kottke_mbcat} for the proof. The first part of this
theorem is proved by observing first that only those
submanifolds that are commensurable with $A$ can enter in the lift of $A$ to the
corresponding blow-up. Blowing up those that are contained in $A$
produces in each case the first factor. Upon blow-up of $A$ itself, we
get either $\ol{V_A}$ in the first case or $\partial V_A$ in the second. The
lift to this of $\partial W$ where $W \supset A$ is just $\partial W_A$, and
this is where the second factor comes from.  The second part follows
by induction.

Returning to Example~\ref{ex.diag.2}, with which we shall be concerned from now on,
it follows that the boundary hypersurfaces of the reduced configuration space $\MB(E^k/E)$ are
in bijection with the quotients
\[
	D_{\lambda k} := D_\lambda/D_k,
\]
where $\lambda$ is a partition of the set $\bk = \set{1,\ldots,k}$ and
$D_\lambda$ is the diagonal in which $p_i = p_j$ whenever $i$ and $j$ lie in
the same block. It is convenient to denote by $0$ the minimal partition of
$\bk$ into $k$ singletons and by $k$ the maximal partition of $\bk$ as a single
set. Then $D_0 = E^k$, $D_k$ is consistent with the earlier definition as the
minimal diagonal, and $D_{0,k} = D_0/D_k$ is the quotient configuration space
$E^k/E$.

We denote the corresponding boundary hypersurface by
\[
	N_\lambda \cong \MB(D_{0 \lambda}) \times \FB(D_{\lambda k}),
\]
and we note in passing the identification
\[
	D_{0 \lambda} = E^k/D_\lambda \cong E^{k_1}/E\times \cdots \times E^{k_n}/E
\]
where $k_1,\ldots,k_n$ are the sizes of the blocks in the partition $\lambda$, while $D_{\lambda k} \cong E^n/E$.

From Theorem~\ref{thm.mbs.bhs} and the fact that diagonal
$D_\lambda$ is contained in $D_\mu$
if and only if $\mu$ refines $\lambda$ (in which
case we write $\mu \leq \lambda$), the codimension $d$ boundary faces of
$\MB(E^k/E)$ have the form
\begin{equation}
	N_{\lambda_1}\cap \cdots \cap N_{\lambda_d} \cong \MB(D_{0\lambda_1})\times \FB(D_{\lambda_1 \lambda_2})\times \cdots \times \FB(D_{\lambda_{d-1} \lambda_d})
	\label{E:reduced_config_faces}
\end{equation}
for totally ordered chains $0 \leq \lambda_1 < \cdots < \lambda_d < k$. Here the many body
structure on $D_{\lambda \mu} = D_\lambda/D_\mu$ is understood to be the set
$\set{D_{\kappa \mu} : \mu \leq \kappa \leq \lambda}$.

In particular, the free boundary $\FB(E^k/E) = N_0$ has as its boundary hypersurfaces
the spaces $N_0 \cap N_\lambda \cong \FB(D_{0 \lambda})$.

\subsection{Divergent Sequences of Configurations}\label{sec.div.sqn}

The motivation for compactifying $E^k$ this way comes from the discussion of
divergent sequences of monopoles in the Introduction.  There we noted
that divergence of a sequence in $\cM_k$ always corresponds to
divergent configurations of points, which are essentially the centers of monopoles of lower charge.  Let
us consider the role of compactifications of $E^k/E$ in handling such
divergent sequences of configurations in the case $k=3$.

If $((p_1^s, p_2^s, p_3^s) : s \in \bN)$ is a divergent sequence in $E^3/E$,
then after passing to a subsequence it has a limit in the interior of some
boundary face of $\MB(E^3/E)$, which we suppose for simplicity has codimension
one. In the case that the mutual separations $\abs{p^s_i - p_j^s}$ all diverge
as $s \to \infty$, then the limit lies on the free boundary $\FB(E^3/E)$, the
interior of which is identified with the sphere $\pa \overline{E^3/E}$.
This limit encodes the relative displacements
$\lim_{s\to\infty}\frac{p^s_i-p_j^s}{\abs{p_i^s-p_j^s}}$ up to overall
translation. (This can be done explicitly, for example, by using translation freedom to set $p_3^s = 0$.)

On the other hand, if one of the separations remains bounded, say $\abs{p_1^s -
p_2^s} < \infty$, then the sequence remains in a neighbourhood of the diagonal
$D_{\lambda k}$, where $\lambda$ is the partition $\set{\set{1,2},\set{3}}$, and the limit lies
on the boundary hypersurface $N_\lambda = \MB(D_{0 \lambda}) \times
\FB(D_{\lambda k})$. We may think of this as the {\em 2 particle cluster} $(p_1^s,
p_2^s)$ diverging from the third particle $p_3^s$.  In this situation, the
sequence of relative configurations $w^s = (w_0^s, w_1^s) = (\tfrac 1 2 (p_1^s +
p_2^s), p_3^s) \in E^2/E$ of their centers of mass converges to a limit in
$\FB(D_{\lambda,k})$, under the identification of $D_{\lambda k}$ with $E^2/E$, while the (recentered) cluster $\mathbf{p}^s = (p_1^s - w_0^s, p_2^s - w_0^s)$
converges to a limit $\mathbf p \in E^2/E \cong D_{0 \lambda}$, the interior of $\MB(D_{0 \lambda})$. Thus
\[
	(p_1^s,p_2^s,p_3^s) \cong (w_0^s, w_1^s, \mathbf{p^s}) \to (w, \mathbf p) \in \FB(D_{\lambda k}) \times \MB(D_{0 \lambda}) = N_\lambda,
\]
where $w = \lim_{s \to \infty} \tfrac{w_0^s - w_1^s}{\abs{w_0^s - w_1^s}}$.

Note that were we to use the radial compactification only, we would only retain the limit
$w$ on the boundary of $D_{\lambda k}$ inside $\pa \overline{E^3/E}$.  The
information about the relative limiting configuration $\mathbf p$ of the 2
particle cluster would be lost.

In our compactification of $\cM_k$, the relative configurations $w^s =
(w_1^s,\ldots,w_n^s)$ of points in $E^n/E$ will be retained, but the role of an
$n$-particle cluster in the preceding discussion will be replaced by a charge $n$
monopole.

\subsection{Unordered Configuration Spaces}\label{sec.sym.mbs}

The symmetric group $\Sigma_k$ acts on $E^k$ and $E^k/E$ by
permutation of the factors.  The quotient spaces are singular, and
we shall not consider them directly.  However, the action is free on
sufficiently small collar neighbourhoods $U$ of $\FB(E^k/E)$ and we need
to understand $U/\Sigma_k$, essentially because the configurations of
points that emerge from divergent sequences of monopoles are
unordered.  More precisely, if a monopole of charge $k_j$ is attached to
$p_j$, then points $p_j$ carrying monopoles of the {\em same} charge
must be regarded as indistinguishable.

Recall
that the boundary hypersurfaces of $\FB : = \FB(E^k/E)$ are labelled by
partitions $\lambda$, $0 < \lambda < k$. For such $\lambda$ we have the
subgroup $\Sigma_\lambda$,
\begin{equation}
\sigma \in \Sigma_\lambda \Leftrightarrow \sigma \mbox{ leaves
}D_\lambda\mbox{ invariant.}
\end{equation}
Then $\Sigma_\lambda$ is the natural goup of symmetries acting on
$(E^k / E)/D_\lambda$ and the stabilizer of the generic point is just
$\{1\}$.
Combinatorially,
\[
	\Sigma_\lambda = \{\sigma \in \Sigma_k : i \sim_\lambda j
        \Leftrightarrow \sigma(i)\sim_\lambda \sigma(j)\}, \;
	\Stab_{\Sigma_k}(D_\lambda) = \{\sigma \in \Sigma_k : \sigma(i)
        \sim_\lambda i\}.
\]
Here we have written $i\sim_\lambda j$ to mean that $i$ and $j$ are
in the same block of $\lambda$.  The group
\begin{equation}\label{def.sigma.lambda}
  \Sym_\lambda = \Sigma_\lambda/\Stab_{\Sigma_k}(D_\lambda)
\end{equation}
is the group of symmetries of the diagonal $D_\lambda$.  Informally,
we think of this as the group of symmetries of configurations of type $\lambda$:
if $\row(\lambda) = n$, then
$\Sigma_n$ acts on the $n$ blocks of $\lambda$ and $\Sym_\lambda$ is
the subgroup in which two blocks can be switched only if they have the
same size.  From this description, it is clear that if $n_j$ denotes
the number of blocks of size $j$, then $\Sym_\lambda$ is the product
of symmetric groups $\Sigma_{n_1}\times \cdots \Sigma_{n_k}$.   In
particular $\Sym_0 = \Sigma_k$ and $\Sym_k = \{1\}$.

Note that the action of $\Sigma_k$ upon $\FB(E^k/E)$ is free, and the
quotient provides a good definition of unordered ideal configurations
of $k$ points in $E$.  We may choose a $\Sigma_k$-invariant product
neighbourhood $U$ of $\FB(E^k/E)$ on which the action is still free (it
suffices to stay away from all diagonals).  Then $U/\Sigma_k$ is a
space of widely-separated (and ideal) unordered configurations of $k$
points in $E$.

The boundary hypersurfaces of $\MB(E^k/E)/\Sigma_k$ and $\FB(E^k/E)/\Sigma_k$ are labelled by
the types $[\lambda]$ of partitions of $k$. In order to describe the
boundary hypersurface $N_a$ of $\MB(E^k/E)/\Sigma_k$ (which are both orbifolds, though their singularities will be of no concern here), pick $\lambda$ such that
$[\lambda]=a$. The subgroup $\Sigma_\lambda$ then acts on the boundary
hypersurface $N_\lambda$ and $N_a = N_\lambda/\Sigma_\lambda$.  This
quotienting can be carried out in two stages, corresponding to the
exact sequence
\begin{equation}\label{e1.21.9.18}
\{1\} \to \Stab_{\Sigma_k}(D_\lambda) \to \Sigma_\lambda \to
\Sym_\lambda \to \{1\}
\end{equation}
and the fibred structure of $N_\lambda$
\begin{equation}\label{e1.8.9.18}
N_\lambda = \MB(D_{0\lambda}) \times \FB(D_{\lambda k}) \to
\FB(D_{\lambda k}),
\end{equation}
(see Theorem~\ref{thm.mbs.bhs}).
It is clear that the subgroup
$\Stab_{\Sigma_k}(D_\lambda)$ of $\Sigma_\lambda$ consists precisely
of those permutations which cover the identity on the base
$\FB(D_{\lambda k})$ and the quotient is
\begin{equation}\label{e11.21.9.18}
(\MB(D_{0 \lambda})/\Stab_{\Sigma_k}(D_\lambda)) \times
\FB(D_{\lambda k}).
\end{equation}
The quotient group $\Sym_\lambda$ in \eqref{e1.21.9.18} is the natural
symmetry group of \eqref{e11.21.9.18}; taking the quotient
gives back $N_a = N_\lambda/\Sigma_\lambda$.

Let us consider the `relative version' of this where $D_{0 k}$ is
replaced by $D_{0 \nu} =D_0/D_\nu$.  Then $\Sigma_k$ is replaced by
$\Sigma_\nu$.  The boundary hypersurfaces of $\MB(D_{0 \nu})$
are labelled by partitions $\kappa$ with $0\leq  \lambda< \nu$.
Define $\Sigma_{\lambda\nu}$ to be the subgroup of $\Sigma_\nu$
which leaves the flag
\begin{equation}\label{e12.21.9.18}
D_\nu \subset   D_\lambda
\end{equation}
invariant, and denote by
$\Sym_{\lambda\nu}$ the quotient
$\Sigma_{\lambda\nu}/\Stab_{\Sigma_k}(D_\lambda)$,
where $\Stab_{\Sigma_k}(D_\lambda)$ is the stablizer of the flag.

When we divide $\MB(D_0/D_\nu)$ by $\Sigma_\nu$, the boundary
hypersurfaces are labelled by the $\Sigma_\nu$-orbits of
partitions $\lambda<\nu$.  Then $\Sigma_{\lambda\nu}$ acts on
the hypersurface $N_{\lambda\mu} = \MB(D_0/D_\lambda) \times
\FB(D_\lambda/D_\nu)$, with quotient the boundary hypersurface
\begin{equation}
N_{[\lambda\nu]} = N_{\lambda\nu}/\Sigma_{\lambda\nu}
\end{equation}
Again, this quotient can be performed in two stages, dividing first by
$\Stab_{\Sigma_k}(D_\lambda)$ and then by $\Sym_{\lambda\nu}$.  The
construction can be generalized in straightforward fashion to corners
of higher codimension in these many-body spaces.

\section{IBF Structures and Compatible \texorpdfstring{$\mathbf{\Phi}$}{Phi}-Metrics}\label{sec.geometry}

In this section we recall some definitions and terminology which
provide a general framework within which we shall describe the
structure of the compactification $\Mon_k$ of $\cM_k$.

The fibred structure of the boundary hypersurfaces of $\Mon_k$ (and
the way the fibrations fit together at the corners) is an example of an
{\em iterated boundary fibration structure} \cite{gpaction}.  Here is the definition:


\begin{dfn}\label{def.mrs} We say that $M$ has an {\em iterated boundary fibration}
    (IBF) if
\begin{enumerate}
\item[(i)] $M$ is a manifold with corners, with boundary hypersurfaces denoted $N_\lambda$, for $\lambda$ in some index set $I$;
\item[(ii)] Every boundary hypersurface $N_\lambda$ of $M$ is equipped with a
  fibration $\phi_\lambda: N_\lambda \to B_\lambda$, where $B_\lambda$ and the fibre $F_\lambda$ of $\phi_\lambda$ are manifolds
  with corners;
\item[(iii)]  If $N_\lambda\cap N_\mu\neq \emptyset$,
then $\dim B_\lambda \neq \dim B_{\mu}$.  If without loss of generality
\footnote{In \cite{gpaction} the inequality is equivalently expressed
in  terms of the dimension of the fibres rather than the dimension of
the base} $\dim B_\lambda > \dim
B_{\mu}$, then $\phi_{\lambda}(N_\lambda\cap N_\mu)$ is a disjoint union of boundary hypersurfaces of $B_{\lambda}$
with full fibre $F_\mu$,
and $\phi_{\mu}$ maps $N_\lambda\cap N_\mu$ surjectively to the base $B_{\mu}$, with fibre a boundary hypersurface (or disjoint union thereof) of $F_{\mu}$.  Finally
there is a fibration $\phi_{\lambda,\mu} : \phi_{\lambda}(N_\lambda\cap N_\mu) \to B_{\mu}$
which satisfies the compatibility condition
\[
  \begin{tikzcd}
    N_\lambda\cap N_\mu \arrow{r}{\phi_{\lambda}} \arrow[swap]{dr}{\phi_{\mu}} &
    \phi_{\lambda}(N_\lambda\cap N_\mu) \arrow{d}{\phi_{\lambda,\mu}} \\
     & B_{\mu}
  \end{tikzcd}
\]
\end{enumerate}
\end{dfn}

\begin{rmk}
We have phrased the definition slightly differently from
\cite[Definition~3.3]{gpaction} by avoiding the notion of `collective
boundary hypersurface'.  Recall that by definition a boundary
hypersurface of a manifold with corners is {\em connected}.  However,
$N_\lambda\cap N_\mu$ need not be connected and so it would be too much to
assume in part (iii) of the definition that
$\phi_\lambda(N_\lambda\cap N_\mu)$ is a single boundary hypersurface of $B_\lambda$.
\end{rmk}

\begin{rmk}
Iterated boundary fibrations arise naturally in resolving smooth group
actions on manifolds \cite{gpaction}, resolving
stratified pseudomanifolds, \cite{sigpack, fibcorn} and in compactification of QALE and QAC
spaces \cite{qac}.   It is the latter applications that are the most relevant
here.  The resolved many-body compactifications of Vasy \cite{vasymb} are also
highly relevant examples, though the above formal definition was not
discussed there.   Unfortunately, the notational conventions and
terminology vary slightly between these references.
\end{rmk}

If $M$ has an IBF structure, then
there is a natural partial order on $I$ defined by the condition:
\begin{equation}\label{eq.I.po-set}
\lambda < \mu \Longleftrightarrow N_\lambda \cap N_\mu \neq\emptyset \mbox{
  and }\dim F_\lambda < \dim F_\mu.
\end{equation}
This ordering gives a notion of `depth', where $N_\mu$ is of greater
depth than $N_\lambda$ if $\lambda < \mu$.   Every corner of $M$ of
codimension $m$ is
then a connected component of an intersection of $m$ boundary
hypersurfaces, and such intersections correspond precisely to a chain
of length $m$ in the partially ordered set $I$.

\begin{rmk}\label{rmk.bhs.mrs} From the definition of IBF, if $N_\mu$ is any
boundary hypersurface, then its boundary hypersurfaces are of two
kinds: Firstly, those that are connected components of $N_\lambda \cap
N_\mu$ with $\lambda < \mu$ fibre over $B_\mu$ and fit into the picture
\begin{equation}
  \begin{tikzcd}
    F_{\lambda\mu} \arrow{r} &  N_\lambda\cap N_\mu \arrow{d}{\phi_{\mu}} \\
     & B_\mu
  \end{tikzcd}
\end{equation}
obtained by restricting $\phi_\mu$ to
a connected component of $N_\lambda \cap N_\mu$ and where $F_{\lambda\mu}$ is a disjoint union of boundary hypersurfaces of $F_\mu$. The other kind is a
connected component of $N_\mu \cap N_\nu$, where $\mu < \nu$ and is
the total space of a connected component of the fibration
\begin{equation}
  \begin{tikzcd}
    F_\mu \arrow{r} &  N_\mu\cap N_\nu
     \arrow{d}{\phi_{\mu}} \\
     & B_{\mu\nu}
  \end{tikzcd}
\end{equation}
again obtained by restricting $\phi_\mu$, where this time $B_{\mu\nu}$ is a disjoint union of boundary hypersurfaces of $B_\mu$.
Thus the boundary hypersurfaces of $F_\mu$ are (connected components
of) the $F_{\lambda\mu}$ with $\lambda < \mu$ and the boundary
hypersurfaces of $B_\mu$ are (connected components of) the
$B_{\mu\nu}$ with $\mu < \nu$, and together they give the boundary hypersurfaces of $N_\mu$. We shall stick to this notation in what follows.

We also observe that Definition~\ref{def.mrs} induces an IBF structure on each fibre $F_\lambda$and base space $B_\lambda$, the boundary hypersurfaces of which are indexed by $\set{\mu \in I :\mu < \lambda}$ and $\set{\mu \in I : \mu > \lambda}$, respectively.
\end{rmk}

\begin{ex}[Many-body Spaces]\label{sec.mbs.mrs}

As an instructive example, let us consider how the many-body compactification $\MB(V^k/V)$ is endowed with an iterated boundary fibration.
As noted above, the boundary hypersurfaces $N_\lambda$ of $\MB(V^k/V)$ are indexed by partitions $\lambda$ and have the form
\[
	N_\lambda = \FB(D_{\lambda k})\times\MB(D_{0 \lambda})
\]
Either of the two factors work as the bases of the boundary fibrations, as long as we make a consistent choice; in light of the construction of $\Mon_k$ (and
as suggested by the notation), we take $\phi_\lambda$ to be the projection onto $B_\lambda := \FB(D_{\lambda k})$, with fibre $F_\lambda := \MB(D_{0 \lambda})$.

If $\lambda < \mu$, then
\[
	N_\lambda \cap N_\mu = B_\mu \times X_{\lambda \mu} \times F_{\lambda} := \FB(D_{\mu k}) \times \FB(D_{\lambda \mu}) \times \MB(D_{0 \lambda}).
\]
The compatibility condition between the fibrations $\phi_\mu$ and $\phi_\lambda$ reads
 \[\begin{tikzcd}
    N_\lambda \cap N_\mu \arrow[r,phantom,"{=}"] &
    B_\mu \times X_{\lambda \mu} \times F_\lambda \arrow{r}{\phi_\lambda} \arrow[start anchor={[xshift=-5ex]},swap]{dr}{\phi_\mu} &
    B_\mu \times X_{\lambda \mu} \arrow{d}{\phi_{\lambda \mu}} \arrow[r, phantom, "{=}"] &
    \phi_\lambda(N_\lambda \cap N_\mu) \\
    & & B_\mu
  \end{tikzcd}\]
and we identify $B_\mu \times X_{\lambda \mu}$ as a boundary face of $B_\lambda$, and likewise $X_{\lambda \mu} \times F_\lambda$ as a boundary
face of $F_\mu$.
\end{ex}

\subsection{Product Structures and Construction of an Adapted Cover}
\label{sec.construction.cover}

\begin{dfn}\label{dfn:bps}  Let $M$ be a compact manifold with IBF
  and boundary hypersurfaces indexed as above. Then
  a {\em boundary product structure} consists of the
  following data.
For each boundary hypersurface $N_\lambda$, an open neighbourhood $U_\lambda$, a smooth
boundary-defining   function $\rho_\lambda$  and a smooth vector field
$v_\lambda$ defined in $U_\lambda$, such that, for any pair $\lambda \neq \mu$,
\begin{equation}
v_\lambda\rho_\lambda = 1\mbox{ but }
v_\lambda \rho_{\mu} =
0 \mbox{ in }U_\lambda\cap U_\mu
\end{equation}
and
\begin{equation}
[v_\lambda,v_\mu] = 0 \mbox{ in }U_\lambda\cap U_\mu.
\end{equation}
Such a boundary product structure is said to be {\em compatible} with the
IBF if for each pair $N_\lambda$, $N_\mu$ of intersecting
hypersurfaces with $\lambda < \mu$,
\begin{eqnarray}
\rho_\mu|N_\lambda &\in& \phi^*_\lambda C^\infty(B_\lambda)\mbox{ near }N_\mu \\
v_\mu|N_\lambda &&\mbox{ is $\phi_\lambda$-related to a vector field on }B_\lambda\mbox{
  near }N_\mu \\
v_\lambda|N_\mu &&\mbox{ is tangent to the fibres of $\phi_\mu$}.
\end{eqnarray}
\end{dfn}

It is shown in \cite[Prop.~1.2, 3.7]{gpaction} that such
compatible boundary product structures always exist.   The argument
is inductive, a key point being that if $M$ is a manifold with
IBF, then in the fibration $\phi_\lambda : N_\lambda
\to B_\lambda$ the base $B_\lambda$ (and the fibre $F_\lambda$)
inherit an IBF structure by virtue of the definition.

 Notice that the
flow of the vector field $v_\lambda$ defines a retraction of $U_\lambda$ to $N_\lambda$ and
so a diffeomorphism from a set of the form $\{\rho_\lambda < \delta\}$ onto
$N_\lambda \times [0,\delta)$.

Let us construct a similar type of compatible boundary structure where
the covering sets $W_\lambda$ are not product neighbourhoods of the
$N_\lambda$ but still of a very useful form:

\begin{prop}\label{thm:ibf.cover} Let $M$ be a compact MWC with IBF and compatible boundary product structure. There is a cover $\{W_\lambda\}$ of a neighbourhood of $\partial M$ so that the restriction of $\phi_\lambda$ to $W_\lambda \cap N_\lambda$ is
surjective, each fibre of $\phi_\lambda$ meets $W_\lambda$ in
a relatively compact subset of $F_\lambda^\circ$ and $W_\lambda \cap
W_\mu = \emptyset$ if $\lambda$ and $\mu$ are not comparable in $I$.
\end{prop}

\begin{proof} Minimal elements in $I$ need not be unique, but the corresponding
boundary hypersurfaces must be disjoint.  If $\lambda$ is a minimal
element of $I$ then the fibre $F_\lambda$ of $\phi_\lambda$ must be a compact
boundaryless manifold by Remark \ref{rmk.bhs.mrs}. For such $\lambda$, we take $W_\lambda$ to be a
product neighbourhood $U_\lambda$ of $N_\lambda$, as per the above
definition.  And naturally we may assume (and will do so) that these $W_\lambda$ are disjoint.

Now let $\Lambda \subset I$ be some subset with the property
$$
\mu \in \Lambda,\; \lambda < \mu \Rightarrow \lambda \in \Lambda.
$$
Suppose that the $W_\lambda$ have been constructed for all $\lambda
\in \Lambda$.   We can now construct $W_{\mu_0}$ for any minimal
element $\mu_0$ of $I\setminus \Lambda$ as follows.  The minimality
condition means that if $\lambda < \mu_0$, then $\lambda \in
\Lambda$.  Hence the intersections $W_\lambda \cap N_{\mu_0}$ for $\lambda
< \mu$ cover the union of the $N_\lambda \cap N_{\mu_0}$ for $\lambda
<\mu_0$ and in particular the entire boundary of each fibre
$F_{\mu_0}$, again by Remark \ref{rmk.bhs.mrs}.  Because the base is compact, we can choose an open
subset $W^\p_{\mu_0}$ of $N_{\mu_0}$ which meets each fibre in a
relatively compact open subset such that the union of $W^{\p}_{\mu_0}$ with all
the intersections $W_\lambda \cap N_{\mu_0}$ covers $N_{\mu_0}$.  Now
use $v_{\mu_0}$ to push $W^{\p}_{\mu_0}$ out into $M$ and denote this
product neighbourhood of $W^{\p}_{\mu_0}$ by $W_{\mu_0}$.  Provided we
don't push out too far, this set will not intersect any $W_\lambda$
with $\lambda$ not comparable to $W_{\mu_0}$.  This completes the
inductive step.
\end{proof}

\subsection{\texorpdfstring{$\mathit{\Phi}$}{Phi}-Tangent Bundle}

Let $M$ be a compact manifold with an IBF. Suppose that
$\{\rho_\lambda\}$ is a collection of compatible boundary-defining functions
and let $\rho$, the total boundary-defining function, be the product of
the $\rho_\lambda$.

Recall that on any MWC $M$,  $\Vect_{\bo}(M)$ is the space of all
smooth vector fields which are tangent to all boundary hypersurfaces.
There is a vector bundle $\bT M \to M$, whose restriction to the
interior is canonically isomorphic to $TM^\circ$, and such that
$C^\infty(M,\bT M) = \Vect_{\bo}(M)$.

\begin{dfn}  The vector field $v \in \Vect_{\bo}(M)$ is called a
{\em $\Phi$-vector field} if
\begin{itemize}
\item $v|N$ is tangent to the fibres of $\phi_N$, or equivalently
  $(\phi_N)_*v =0$ for every boundary hypersurface $N$;
\item $v(\rho) \in \rho^2C^\infty(M)$.
\end{itemize}
Roughly speaking, $v$ is bounded in the fibre directions and scales in an asymptotically conic
fashion in the base directions.
The space of all $\Phi$-vector fields will be denoted by $\Vect_{\Phi}(M)$ and as in the case of $b$-vector fields there is a $C^\infty$ vector bundle $\fT M$ with the property that $C^\infty(M,\fT M) = \Vect_{\Phi}(M)$.
\end{dfn}

\begin{rmk}
The first item does not depend upon the choice of $\rho$, but the
second one does. $\Vect_\Phi(M)$ is the same as the
algebra of vector fields defined in \cite{qac}, we have given a
different but equivalent definition.  In \cite{fibcorn} there is a
closely related definition which is, however, {\em not} equivalent.

We have chosen to call these $\Phi$-vector fields over `QFB' vector
fields because the present notion is the natural generalization of the
initial use of `$\Phi$' in Mazzeo--Melrose \cite{MM_FB} to manifolds with
fibred boundary.  When $M$ has a single boundary hypersurface
$N$, we recover the original notion.
\end{rmk}

\begin{dfn}  Given a compact manifold with IBF, a
  $\Phi$-metric is a smooth (up to all boundary hypersurfaces) metric
  on the bundle $\fT M$.
\end{dfn}

\begin{rmk}
It follows easily from the definition that
the restriction of a $\Phi$ vector field (resp. $\Phi$-metric) to a fibre $F_\lambda$
of a boundary hypersurface of $M$ is again a $\Phi$ vector field (resp.\ $\Phi$-metric)
on $F_\lambda$, with respect to its IBF structure induced from that on $M$.
\end{rmk}

\begin{prop}\label{prop.mrs.bounded.cover}
Suppose that $M$ is a compact manifold with IBF and let $g$ be a compatible $\Phi$-metric on $M$.
Then, denoting the boundary fibrations etc.~as before, there exists a finite open cover $\{V_0\} \cup \{V_\lambda\}_{\lambda \in I}$ of $M$ and a partition of unity $\chi_\lambda$ subordinate to this cover so that, for $\lambda \neq 0$, $\phi_\lambda(V_\lambda \cap \p M)=B_\lambda$, $V_\lambda \cap \phi_\lambda^{-1}(\{p\}) \subset F_\lambda^\circ$ is relatively compact, and such that all $|\nabla \chi_\lambda|_g$ are uniformly bounded.
\end{prop}

We have already constructed a cover $\{W_\lambda\}$ of $\p M$ with the some of the required properties, cf.\ Proposition \ref{thm:ibf.cover}. Thus we only need to extent this to cover all of $M$ and show that any partition of unity subordinate to such a cover has uniformly bounded derivatives. But this does not depend on the specific choice of cover and is true on more general grounds:

\begin{proof} So let $M$ be a compact MWC with IBF, compatible $\Phi$-metric $g$ and boundary fibrations denoted as before, i.e., $\phi_\lambda : N_\lambda \longrightarrow B_\lambda$ with fibre $F_\lambda$ where $\lambda \in I$ indexes the boundary hypersurfaces of $M$. Suppose $\{W_\lambda\}$ is the cover of $\p M$ constructed in Section \ref{sec.construction.cover}. Then, the sets $W_\lambda$ already satisfy $\phi_\lambda(W_\lambda \cap \p M) = \phi_\lambda(W_\lambda^\p) = B_\lambda$ and the sets $\phi_\lambda^{-1}(\{p\}) \cap W_\lambda$ are relatively compact in $F_\lambda^\circ$. Thus, the cover is global in the base and local in the fibres. Finally, let $V_\lambda = W_\lambda$, for $\lambda \in I$ and $V_0$ be an open neighbourhood of $M \setminus \bigcup_{\lambda \in I} W_\lambda$.

Now let $\chi_\lambda$ be a smooth partition of unity subordinate to $\{V_0\} \cup \{V_\lambda\}_{\lambda \in I}$. Since the $\chi_\lambda$ are smooth, so too are the $1$-forms $\rd \chi_\lambda$.  Any smooth $1$-form is also a smooth section of $\fT^*M$, hence the functions $|\rd \chi_\lambda|^2_g : M \longrightarrow \RR$ are smooth, compactly supported and there are only finitely many of them. Thus, they are uniformly bounded on $M$.
\end{proof}

The significance of this result is that such good covers with controlled partitions of unity are required in the proof of the coprime case of the Sen Conjecture presented in \S\ref{psen}. In particular, they allow us to use a Mayer-Vietoris argument to compute the smooth-$L^2$-de Rham cohomology of $M$.

\begin{rmk} Using the specific structure of the cover $\{V_\lambda\}$ and of $\fT^*M$, we could in fact deduce more precise information about the behaviour of $|\rd \chi_\lambda|^2_g$ near $\p M$. But this amount of detail will not be needed in what is to follow.
\end{rmk}

\section{The Moduli Space Compactification}\label{sec.mon.cpctn}
\subsection{Monopoles}\label{sec.background}

We recall the precise definition of the moduli space of framed
euclidean monopoles of charge $k$, from a point of view that makes
subsequent generalizations natural.

First of all, we shall use $\cC$ for the space of `monopole data' on a
radially compactified euclidean 3-dimensional euclidean space $E =
\ol{\RR}^3$:   thus $\cC$ consists of pairs $(A,\Phi)$ where $A$ is
an $\SU(2)$ connection on a bundle $P\to E$ and $\Phi$ is a
smooth section of the adjoint bundle. Here smooth means `up to and
including the boundary': to say $A$ is smooth is to say that in any
smooth local trivialization of $P$ the connection $1$-form is smooth
(including neighbourhoods of boundary points of $E$.  We
shall use obvious variations of the notation such as
$\cC(E,P)$ if we need to make the base space or bundle
explicit.   The gauge group $\cG$ consists of automorphisms of $P$
which again are smooth over $E$.  This group is the natural
infinite-dimensional symmetry group of $\cC$: it acts by pull-back.

The (euclidean) Bogomolny equations on $\cC$ are
\begin{equation}\label{e1.30.8.18}
\cB(A,\Phi) = \nabla_A\Phi - *F_A,
\end{equation}
where $F_A$ is the curvature of $A$ and $*$ is the euclidean Hodge
star operator. The equations are gauge-invariant and a naive
definition of the monopole moduli space would be to divide the zeros
of $\cB$ by the action of $\cG$.

The framed moduli space \cite{AH,KS} is a refinement in which we fix
$(A,\Phi)$ up to and including $O(\rho)$ terms, where $\rho$ is the
standard defining function of the boundary of the radial
compactification (reciprocal of distance from $0$).  We denote by
$(A_{\bo},\Phi_{\bo})$ these fixed data defined near the boundary and
define
\begin{equation}
\cC_{\Fr} = \{(A,\Phi) \in \cC : (A-A_{\bo},\Phi - \Phi_{\bo}) =
O(\rho^2)\}
\end{equation}
where we use the euclidean metric to measure the length of the
$1$-form $A-A_{\bo}$.   There is a corresponding framed gauge group
$\cG_{\Fr}$ consisting of those gauge transformations which preserve
$\cC_{\Fr}$ and equal to the identity on the boundary.  The framing
data $(A_{\bo},\Phi_{\bo})$ are essentially given by an abelian
monopole, and hence there is an associated topological charge $k$, the
winding number of $\Phi_{\bo}$.  We shall denote a framing of charge
$k$ by $\Fr_k$.

Consider the $1$-parameter subgroup $\gamma_t = \exp(t\Phi/2)$ of gauge
transformations. Then  $\gamma_t \in \cC_{\Fr}$.  With our factor of
$2$, $\gamma_{2\pi}|\p E=-1$ and so the 
conjugation action on $\cC$ is trivial at
the boundary. Thus $g_{2\pi}$ acts as the identity on the framed
moduli space.  So with the factor of $2$, we have an action of
$\TT = \RR/2\pi \ZZ$ on the the framed moduli space (cf.\ \cite{AH}).

\begin{dfn}
The framed moduli space of euclidean charge-$k$ monopoles is defined
as follows
\begin{equation}
\cN_k = \{(A,\Phi) \in \cC(E,P;\Fr_k) : \cB(A,\Phi)=0\}/\cG_{\Fr}.
\end{equation}
\end{dfn}

\begin{rmk}
Any two choices of framing are gauge-equivalent (though not
  by an element of $\cG_{\Fr}$).  Thus different choices of framing in
  the definition of $\cN_k$ lead to diffeomorphic framed moduli spaces.
\end{rmk}

\subsubsection{Properties of $\cN_k$}

It is known that $\cN_k$ is a smooth manifold of dimension $4k$.  It
carries a natural $L^2$ metric, $g_k$, say, which is complete and
hyperK\"ahler. The reader is referred to
\cite{AH} and references therein for these standard facts.

On $\cN_k$ there are important isometries: those induced by 
translations of $\RR^3$ and the action of $\TT$ by `frame rotation',
induced by the $1$-parameter group $t \mapsto \gamma_t$
described above.  Accordingly there is a reduced
moduli space $\cM_k=\cN_k/\RR^3$ of dimension $4k-3$.  Factoring out
by the translations is essentially the same as
restricting to monopoles with \emph{centre} at $0 \in \RR^3$. One
definition of the centre (of mass) of a monopole is in terms of the
$\TT$-action on $\cN_k$. This is a triholomorphic isometry with a
hyperK\"ahler moment map $c : \cN_k \longrightarrow \RR^3$. This map
is a submersion and $c^{-1}(0)$ is the submanifold of monopoles
centred at $0$. If $m \in \cN_k$, there is a unique $p \in \RR^3$ so
that $m(\cdot-p) \in c^{-1}(0)$. Thus the quotient $\cM_k$ and
$c^{-1}(0) \subset \cN_k$ are essentially interchangeable.

This second point of view shows that the \emph{reduced} moduli space
 \begin{equation}
  \cMo_k = \cN_k /\!\!/\!\!/ \TT = c^{-1}(0)/\TT
 \end{equation}
is again hyperK\"ahler, of dimension $4k-4$.
We can equally define $\cMo_k = \cN_k /\big(\RR^3 \times \TT\big)$, for the two actions of $\RR^3$ and $\TT$ on $\cN_k$ commute with each other.
It is known that $\pi_1(\cMo_k) = \ZZ_k$ (cf.\ \cite{AH}); we denote the universal cover by $\unicov{k}$.
The cover $\unicov{k}$ is called the space of \emph{strongly centred monopoles (of charge $k$)}.

A famous result of Donaldson \cite{skd1984,AH} shows that, after a
choice of direction in $\RR^3$, $\cN_k$ can be identified with the
space $\Rat_k$ of based rational maps of degree $k$ of $\CC P^1$ to
itself.  The basing condition restricts us to rational maps of the
form
$$
f(z) = \frac{\phi(z)}{\psi(z)}
$$
where
$$
\phi(z) = a_0 + \cdots + a_{k-1}z^{k-1},\;\;
\psi(z) = b_0 + \cdots + b_{k-1}z^{k-1} + z^k,\;\;
$$
are complex polynomials with no common factors. The coefficients
$(a_0,\ldots, a_{k-1},b_0,\ldots, b_{k-1})$ are coordinates on
$\cN_k$.  The condition that $\phi$ and $\psi$ have no common factors
is equivalent to the non-vanishing of the resultant
$$
\cR(\phi,\psi) = \prod_j \phi(\beta_j)
$$
where the $\beta_j$ are the roots of $\psi$.  $\cR$ is homogeneous of
degree $k$ in the coefficients $(a_0,\ldots,a_{k-1})$.  With this
description, $\cM_k$ is the subspace of rational maps with
$$
b_{k-1} = 0 , |\cR(\phi,\psi)| =1.
$$
The $\TT$-action is just $\phi \mapsto \lambda \phi$, $|\lambda|=1$.
The strongly centred space $\unicov{k}$ is then given by the conditions
$$
b_{k-1} = 0 , \cR(\phi,\psi) =1.
$$
The subgroup $\ZZ_k\subset \TT$ preserves these conditions
because of
the observation about the homogeneity of $\cR$ in the coefficients of
$\phi$. This $\ZZ_k$ is the group of deck transformations, and factoring out
by it gives the space $\unicov{k}/\ZZ_k = \cMo_k = \cM_k/\TT$.

Note that the metric is not easy to describe in this picture.  As
observed by Atiyah and Hitchin in \cite[p.\ 19]{AH}, the description
of $\cM^0_k$ in terms of rational maps exhibits it as a dense open subset
of $\CC P^{k-1}\times \CC P^{k-1}$.  This observation certainly gives
a compactification of $\cM_k^0$, but this will be different from ours
and is unlikely to have good properties with respect to the monopole
metric.

\subsection{Ideal Monopoles and the Boundary Hypersurfaces of
  \texorpdfstring{$\Mon_k$}{M(k)}}\label{sec.mon.bhs}

We now describe the structure of the compactification $\Mon_k$ of
$\cM_k$ in detail.  For each partition $a = (k_1,\ldots,k_n)$ of the
integer $k$, there is a boundary hypersurface $\cI_a$, and $\cI_a$
meets $\cI_b$ if and only if the partitions $a$ and $b$ are
comparable.   In order to describe $\cI_a$ it is best to choose a
partition $\nu$ of $\bk$ whose type is $a$ (so that the sizes of the
blocks of $\nu$ are the integers $k_1,\ldots, k_n$).  Then we have
\begin{equation}\label{e11.7.10.18}
\cI_a = \cI_\nu/\Sym_\nu.
\end{equation}
It is natural to think of $\cI_\nu$ as an `ordered version' of $\cI_a$.   The
ingredients needed to define $\cI_\nu$ are a compactification
$\Mon_\nu$ of the product
\begin{equation}\label{e12.7.10.18}
\cM_\nu = \cM_{k_1}\times  \cM_{k_2} \times \cdots \times \cM_{k_n}.
\end{equation}
and a certain rank-$\row(\nu)$ torus-bundle
\begin{equation}\label{e13.7.10.18}
\cT_{\nu k}  \longrightarrow \FB_{\nu k},
\end{equation}
where $\FB_{\nu k}$ is the boundary of the free region of
$\MB(D_{\nu}/D_k)$ as before.  This torus-bundle has the property that
it can be chosen to admit an action of $\Sym_{\nu}$ which permutes the
factors and covers the action of $\Sym_{\nu}$ on
$\FB_{\nu,k}$. Similarly, the $\TT^{\row(\nu)}$- and permutation-action of
$\Sym_{\nu}$ on $\cM_\nu$ extend smoothly to $\Mon_\nu$.  Using the
$\TT^{\row(\nu)}$-action, we define
\begin{equation}\label{e14.7.10.18}
\cI_\nu = \cT_{\nu k} \times_{\TT^n} \Mon_{\nu};
\end{equation}
this space inherits an action of $\Sym_\nu$, allowing us to define
$\cI_a$ as the quotient \eqref{e11.7.10.18}.   The torus-bundle
$\cT_{\nu k}$ appeared in \cite{KS} and in Bielawski's work and is
called a (generalized) Gibbons--Manton bundle; the version for the
`free' partition $(1,\ldots,1)$ appears in the original paper
\cite{gibbons1995} in the description of this asymptotic region of $\cM_k$.  We
shall recall the definition in a moment, but pause first to note that
\eqref{e14.7.10.18} is clearly incomplete without a definition of the
compactification $\Mon_{\nu}$.  This, however, has a description very
analogous to that of $\Mon_k$ itself, where the boundary hypersurfaces
are now finite quotients of spaces
\begin{equation}\label{e15.7.10.18}
\cI_{\lambda\nu} = \Mon_{\lambda} \times_{\TT^{\row(\lambda)}}
\cT_{\lambda\nu}
\end{equation}
where now $0 \leq \lambda < \nu$ and $\cT_{\lambda\nu}$ is a
  Gibbons--Manton torus bundle of rank $\row(\nu)$ over
  $\FB_{\lambda\nu}$, the boundary of the free region of
  $\MB_{\lambda\nu} =\MB(D_{\lambda}/D_{\nu})$.  Since
  \eqref{e15.7.10.18} only involves  the spaces $\Mon_{\lambda}$ with
  $\lambda<\mu$, these spaces can be
  built up inductively starting from the free partition
\begin{equation}\label{e16.7.10.18}
\Mon_0 = \cM_0 = \cM_1 \times \cdots \times \cM_1 = (S^1)^k
\end{equation}
and ending with $\Mon_k$.

In the remainder of this section, we fill in the details, summarise
the properties of the space $\Mon_\mu$ that are needed to make the
induction work.  We also verify that the collection of manifolds
$\{\cI_{\lambda\mu} : 0 \leq \lambda < \mu\}$ satisfy the
compatibility conditions needed for $\Mon_\mu$ to be a compact MWC
with IBF.

\subsection{Gibbons--Manton Bundles}  Let $\lambda < \nu$ be two
partitions of the set $\bk$.  The Gibbons--Manton bundle $\cT_{\lambda\nu}$ is defined initially
over the space
\begin{equation}\label{e20.7.10.18}
E_{\lambda\nu} = D_{\lambda}/D_{\nu} \setminus \bigcup_{\lambda < \nu
  \leq \nu} D_\nu/D_\nu.
\end{equation}
Recall that $D_\lambda \subset (E^\circ)^k$, where $E$ is our fixed
radially compactified $3$-dimensional euclidean space.  For each $1\leq i,j \leq k$, define
the difference map
\begin{equation}\label{e21.7.10.18}
\pi_{ij} = p_i - p_j\mbox{ so }\pi_{ij}: (E^\circ)^k \to E^\circ.
\end{equation}
 Notice that $\pi_{ij}$ vanishes on $D_\nu$ if and only if
 $i\sim_{\nu} j$.  Therefore,
\begin{lem}
The difference map $\pi_{ij}$ is induces a map on $E_{\lambda\nu}$ if
  and only if $i\sim_\nu j$, to be denoted by the same symbol. This
  induced map is non-zero on $E_{\lambda\nu}$ if and only if
  $i\not\sim_\lambda j$ (but $i\sim_{\nu} j$).
\end{lem}

Now denote by $\omega$ the $\mathrm{SO}(3)$-invariant closed $2$-form on
$E^\circ\setminus 0$ whose de Rham class $[\omega]$ generates
$H^2(E^\circ\setminus 0,\ZZ)$.   For $1\leq i \leq k$, define
\begin{equation}\label{e22.7.10.18}
\omega_i = 2\sum_{i\sim_\nu j,\; i\not\sim_{\lambda} j}
\pi_{ij}^*(\omega).
\end{equation}
Then $\omega_i$ is a closed integral $2$-form on $E_{\lambda\nu}$ and
one can find a circle-bundle $Q_i$ with connection $\alpha_i$ such
that $\rd \alpha_i = 2\pi \sqrt{-1} \omega_i$.  Since $E_{\lambda\nu}$
is simply connected, $(Q_i,\alpha_i)$ is unique up to isomorphism.
From the definition of $\omega_i$, it is clear that $\omega_i =
\omega_{i'}$ if $i\sim_\lambda i'$, for the induced maps $\pi_{ij}$
and $\pi_{i'j}$ on $E_{\lambda\nu}$ are then equal.  So pick a set of
indices $i_1,\ldots, i_\ell$ ($\ell = \row(\lambda)$) such that $i_j$ is in the $j$-th block of
$\lambda$ and define
\begin{equation}\label{e23.7.10.18}
\cT_{\lambda\nu} = Q_{i_1} \times \cdots \times Q_{i_\ell}.
\end{equation}
This is the (generalized) Gibbons--Manton bundle (of type
$\lambda\nu$).  Note that its rank is equal to $\row(\lambda)$.  We
note the following:
\begin{lem}
As defined above, $\cT_{\lambda\nu}$ extends uniquely from
$E_{\lambda\nu}$ to the free region of $\MB_{\lambda\nu}$ and in
particular to $\FB_{\lambda\nu}$.
\end{lem}
\begin{proof} This can be seen by noting that $\pi_{ij}$ extends to
  define a smooth map $\MB_{\lambda\nu} \to E$, which is non-zero
  on the free region, provided, of course, that
$$
i\sim_\nu j,\; i\not\sim_{\lambda} j.
$$
The form $\omega$ also extends smoothly to $E\setminus 0$.  Thus
$\omega_i$, as defined in \eqref{e22.7.10.18}, extends to define a
smooth $2$-form on the free region of $\MB_{\lambda\nu}$.  Knowing
this, it follows that the $Q_i$ admit smooth extensions to this free
region as well.
\end{proof}

For the compatibility conditions between the different
$\cI_{\lambda\nu}$ we shall need to understand the restriction of
$\cT_{\lambda\nu}$ to a boundary hypersurface of
$B_{\lambda\nu}$. Recall that these boundary hypersurfaces are indexed
by partitions $\mu$ with $0 \leq \mu < \nu$.  The $\mu$-boundary
hypersurface will be denoted by $\p_{\mu}\FB_{\lambda\nu}$, and we have
seen (Theorem~\ref{thm.mbs.bhs}) that
\begin{equation}\label{e1.13.10.18}
\p_{\mu}\FB_{\lambda\nu} = \FB_{\lambda\mu}\times \FB_{\mu\nu}.
\end{equation}
There are two natural torus-bundles over this space.  One is the
restriction of $\cT_{\lambda\nu}$.  For the other, note that over the
product we have $\cT_{\lambda\mu}$ and $\cT_{\mu\nu}$, respectively of
ranks $\row(\lambda)$ and $\row(\mu)$. Here
$\row(\mu)< \row(\nu)$ and there is a canonical inclusion $\TT^{\row(\mu)}
\hookrightarrow \TT^{\row(\lambda)}$ corresponding in the obvious way
to the inclusion $D_\mu \hookrightarrow D_{\lambda}$.   Thus
$\TT^{\row(\mu)}$ acts on both $\cT_{\lambda\mu}$ and $\cT_{\mu\nu}$
and we may define the rank-$\row(\lambda)$ torus-bundle
\begin{equation}\label{e2.13.10.18}
\cT_{\lambda\mu} \times_{\TT^{\row(\mu)}} \cT_{\mu\nu} \to
\FB_{\lambda\mu}\times \FB_{\mu\nu}.
\end{equation}

The key result is as follows:
\begin{lem}\label{lem.GM.bdy}  When $\p_\mu \FB_{\lambda\nu}$ is identified with the
  product as in \eqref{e1.13.10.18}, we have
\begin{equation}\label{E:GM_boundary_ident}
\cT_{\lambda\nu}|\p_\mu\FB_{\lambda\nu} \simeq\cT_{\lambda\mu}
\times_{\TT^{\row(\mu)}} \cT_{\mu\nu} \mbox{ over }
\FB_{\lambda\mu}\times \FB_{\mu\nu}.
\end{equation}
\end{lem}

\begin{proof}  Starting from the identification
$D_{\lambda,\nu} \cong D_{\lambda,\mu} \times D_{\mu,\nu}$, we obtain
the relation
\[  E_{\lambda\nu} \hookrightarrow E_{\lambda\mu} \times E_{\mu\nu}
\]
between  the free regions of $D_{\lambda\nu}$, $D_{\lambda\mu}$ and
$D_{\mu\nu}$,  since only diagonals of the form $D\times E_{\mu\nu}$
and $E_{\lambda\mu}\times D'$ are removed on the right-hand side.

Then, fixing a block of $\lambda$ with representative element $i$, and
working complex line bundles $L_i$ instead of the circle-bundles
$Q_i$, we have an
\[
L_{i,\lambda\nu} \cong  L_{i,\lambda\mu} \otimes L_{i,\mu\nu};
\]
this follows at the level of Chern classes by splitting the sum
\eqref{e22.7.10.18} defining the LHS according as $j$ is or is not in
the same $\mu$-block of $i$.  As before, this
extends to an isomorphism over the boundary hypersurface
$\p_{\mu}\FB_{\lambda\nu}$ as required.
\end{proof}

This completes our discussion of the generalized Gibbons--Manton
bundles. Now that these are defined, \eqref{e15.7.10.18} makes sense
for any space $\Mon_{\mu}$ with a free $\TT^{\row(\mu)}$-action.

\subsection{Main Theorem}\label{sec.main.thm}

We are now nearly ready to state the main result, the proof of which
will appear in Part II.  First recall the definitions
$\Sigma_{\lambda\nu}$, $\Sym_{\nu}$ from \S\ref{sec.sym.mbs}.  Define also, for
any pair $\lambda < \nu$,
\begin{equation}\label{e5.13.10.18}
\Sym_{\lambda\nu}^0 =
\frac{\Stab_{\Sigma_{\lambda\nu}}(D_\nu)}{\Stab_{\Sigma_k}(D_\lambda)},\;\;
\Sym_{\lambda\nu} =
\frac{\Sigma_{\lambda\nu}}{\Stab_{\Sigma_k}(D_\lambda)}
\end{equation}
so we have the exact sequence
\begin{equation}
\{1\} \to \Sym_{\lambda\nu}^0 \to \Sym_{\lambda\nu} \to
\frac{\Sigma_{\lambda\nu}}{\Stab_{\Sigma_k}(D_\nu)} \to \{1\}.
\end{equation}
Both $\Sym^0_{\lambda\nu}$ and $\Sym_{\lambda\nu}$ are symmetry groups
of the flag $D_\nu \subset D_{\lambda}$, the former being the subgroup
of permutations equal to the identity on $D_\nu$.

We note that these groups act on the set of $2$-forms
$\{\omega_{i_1},\ldots,\omega_{i_\ell}\}$, so that $\sigma^*(\omega_i)
= \omega_{\sigma^{-1}(i)}$.  There is a corresponding lift of these
group actions from $\FB_{\lambda\nu}$ to the torus-bundle
$\cT_{\lambda\nu}$.

\begin{thm}\label{cpct.conj} Let $k>1$ and let $\nu$ be a partition of
  $\bk$.  Then there is a compactification $\Mon_{\nu}$ of the product
  $\cM_{\nu}$ as a manifold with iterated boundary fibration structure
  having the following properties:
\begin{enumerate}
\item The boundary hypersurfaces of $\Mon_{\nu}$ are indexed by the
$\Stab_{\Sigma_k}(D_\nu)$-orbits of
partitions $\lambda$ with $0\leq \lambda < \nu$. Given $\lambda<\nu$,
the corresponding boundary hypersurface is
\begin{equation}
N_{\lambda\nu} = \cI_{\lambda\nu}/\Sym^0_{\lambda\nu},
\end{equation}
where
\begin{equation}
\cI_{\lambda\nu} = \Mon_{\lambda} \times_{\TT^{\row(\lambda)}}
\cT_{\lambda\nu}.
\end{equation}
In particular, with $\nu$ fixed, for each $\lambda$ we have a fibration
$$\phi_\lambda:
N_{\lambda\nu} \to B'_{\lambda\nu},$$
where
$$
B'_{\lambda\nu} = \FB_{\lambda\nu}/\Sym^0_{\lambda\nu}.
$$
\item  If $\lambda < \mu < \nu$, then the intersection $N_{\lambda\nu}
  \cap N_{\mu\nu}$ is non-empty, and all such intersections arise in
  this way up to the action of $\Sym^0_{\lambda\nu}$.  In this case,
there exists
  $\phi_{\lambda\mu}$ giving the compatibility conditions of an IBF
  structure (cf.\ Definition~\ref{def.mrs}).
\item The $\TT^{\row(\nu)}$- and $\Sym_{\nu}$-actions extend
  smoothly from $\cM_{\nu}$ to $\Mon_{\nu}$.  The quotient
  $\Mon_{[\nu]}=\Mon_{\nu}/\Sym_{\nu}$ has boundary hypersurfaces
$\cI_{[\lambda\nu]}$, say, indexed by the $\Sym_{\nu}$-orbits $[\lambda\nu]$ of partitions
  $\lambda<\nu$, and
$$
\cI_{[\lambda\nu]} = \cI_{\lambda\nu}/\Sym_{\lambda\nu}
$$
\end{enumerate}

\end{thm}

This will not be proved here, but we shall carry out the consistency
checks that are implied by it.  In particular we shall check that the
definitions of the boundary hypersurfaces are consistent with the
points enumerated in the Theorem.

\begin{rmk} In the statement of
this Theorem we abuse notation: we label the boundary hypersurfaces by
partitions $\lambda<\nu$, where really the labelling is by the
$\Stab(D_{\nu})$-orbits of such partitions.  The only difficulty with
this is that we `overcount' the boundary hypersurfaces this way:
$N_{\lambda\nu}$ and $N_{\lambda'\nu}$ are the same boundary
hypersurface of $\Mon_{\nu}$ if and only if $\lambda' =
\sigma(\lambda)$ for some $\sigma \in \Stab_{\Sigma_k}(D_\nu)$.
\end{rmk}

Let us start with a check on dimensions. The dimension of $\cM_{k}$ is
$4k-3$, so
$$
\dim \cM_\lambda = 4k - 3\row(\lambda).
$$
Similarly $\dim D_\lambda = 3\row(\lambda)$.   Hence the dimension of
$\FB_{\lambda\nu}$, being a boundary hypersurface of
$\MB_{\lambda\nu}$, is $3(\row(\lambda)-\row(\nu)) - 1$.  Thus
$$
\dim \cI_{\lambda\nu} = 4k - 3\row(\lambda) +
3(\row(\lambda)-\row(\nu)) - 1 = 4k - 3\row(\nu) -1.
$$
Moreover, if $\lambda< \mu$ then $\dim B_{\lambda\nu} > \dim
B_{\mu\nu}$ so our ordering conventions are consistent with those used
in Definition~\ref{def.mrs}.

We now come to the main point, the intersections of the boundary
hypersurfaces of $\Mon_{\nu}$.  Notice that $N_{\lambda\nu}$ has two
types of boundary hypersurface: from the inductive description, there
are those corresponding to partitions $\mu$ with
\begin{equation}\label{e11.14.10.18}
D_\nu \subset D_\mu \subset D_\lambda
\end{equation}
and those corresponding to partitions $\kappa$ with
\begin{equation}\label{e12.14.10.18}
D_\mu \subset D_\lambda \subset D_\kappa.
\end{equation}
Fixing the chain \eqref{e11.14.10.18}, consider the $\mu$ boundary
hypersurface of $N_{\lambda\nu}$ and the $\lambda$-boundary
hypersurface of $N_{\mu\nu}$.  Now
the $\mu$-boundary surface of
$\cI_{\lambda\nu}$ is just the restriction of the fibration
$$
\Mon_\lambda \times_{\TT^{\row(\lambda)}} \cT_{\lambda\nu} \to
\FB_{\lambda\nu}
$$
to the the $\mu$-boundary hypersurface of the base,
\begin{equation}\label{e3.14.10.18}
\p_\mu \FB_{\lambda\nu} = \FB_{\lambda\mu} \times \FB_{\mu\nu}
\end{equation}
To take into account the group action, we must factor out by the
subgroup $G$, say, of $\Sym^0_{\lambda\nu}$ which leaves
\eqref{e11.14.10.18} invariant.  Thus
\begin{equation}\label{e2.14.10.18}
G = \Stab_{\Sigma_{\lambda\mu\nu}}(D_\nu)/ \Stab(D_\lambda)
\end{equation}
where $\Sigma_{\lambda\mu\nu}$ is the group of all permutations in
$\Sigma_k$ which leave \eqref{e11.14.10.18} invariant.
Using Lemma~\ref{lem.GM.bdy}, we obtain
\begin{equation}\label{surface1}
\p_\mu N_{\lambda\nu} =
\left(\Mon_{\lambda}
\times_{\TT^{\row(\lambda)}} \cT_{\lambda\mu}
\times_{\TT^{\row(\mu)}} \cT_{\mu\nu}\right)/G \to
(\FB_{\lambda\mu} \times \FB_{\mu\nu})/G.
\end{equation}
On the other hand, the $\p_\lambda N_{\mu\nu}$ is obtained by
restricting $\phi_\mu$ to the $\lambda$-boundary in the fibres, that
is
$$
N_{\lambda\mu} \times_{\TT^{\row(\mu)}} \cT_{\mu\nu}/G'
$$
Here $G'$ is the subgroup of $\Sym^0_{\mu\nu}$ which leaves $D_\lambda$
invariant,
$$
G' = \frac{\Stab_{\Sigma_{\lambda\mu\nu}}(D_\nu)}{
\Stab_{\Sigma_{\lambda\mu\nu}}(D_\mu)} =
\frac{\Stab_{\Sigma_{\lambda\mu\nu}}(D_\nu)}{\Stab_{\Sigma_{\lambda\mu}}(D_\mu)}
$$
By the inductive assumption,
$$
N_{\lambda\mu} = \Mon_{\lambda} \times_{\TT^{\row(\lambda)}}
  \cT_{\lambda\mu}/\Sym^0_{\lambda\mu},
$$
so
\begin{equation}\label{surface2}
\p_\lambda N_{\mu\nu}
=
\left(
\frac{\Mon_{\lambda} \times_{\TT^{\row(\lambda)}}
  \cT_{\lambda\mu}}{\Sym^0_{\lambda\mu}} \times_{\TT^{\row(\mu)}}
  \cT_{\mu\nu}
\right)/G'.
\end{equation}
If we ignore the group actions, we see that we have the same manifolds
in \eqref{surface1} and \eqref{surface2}, as both are equal to
\begin{equation}
\Mon_{\lambda} \times_{\TT^{\row(\lambda)}}
  \cT_{\lambda\mu} \times_{\TT^{\row(\mu)}}
  \cT_{\mu\nu} \to \FB_{\lambda\mu} \times \FB_{\mu\nu}
\end{equation}
So it remains only to see that the successive quotients first by
$\Sym^0_{\lambda\mu}$ and then $G'$ in \eqref{surface2} are equivalent
to factoring out by $G$  in \eqref{surface1}.  From the descriptions
of $G$ and $G'$, however, we see that $G'$ is a quotient of $G$, and
in fact
\begin{equation}
\{1\} \to
\Sym^0_{\lambda\mu} \to
{\Stab(D_\lambda)} \to
G \to G' \to \{1\}.
\end{equation}
From this it
follows that \eqref{surface1} and \eqref{surface2} are naturally
diffeomorphic.

Identifying the intersection of boundary hypersurfaces corresponding
to the chain \eqref{e11.14.10.18} with \eqref{surface1}, the
restriction of $\phi_\lambda$ is the projection map given there.  The
second projection $\FB_{\lambda\mu} \times \FB_{\mu\nu} \to
\FB_{\mu\nu}$ induces a map
$$
\phi_{\lambda\mu} : \FB_{\lambda\mu} \times \FB_{\mu\nu}/G \to
\FB_{\mu\nu}/\Sym^0_{\mu\nu}
$$
and we clearly have $\phi_\mu = \phi_{\lambda\mu}\circ\phi_\lambda$ as
required by Definition~\ref{def.mrs}.  Thus, although we haven't yet
proved that the compactification $\Mon_{\nu}$ of $\cM_{\nu}$ exists,
if we assume that all $\Mon_{\lambda}$, for $\lambda<\nu$ have been
constructed with the above properties, then we can form a collection
of MWCs, namely the $\cI_{\lambda\nu}$ and their quotients the
$N_{\lambda\nu}$, which fit together to form a `formal boundary'
with IBF structure.

We have not yet discussed the group actions, point (3) of the
above.  On the boundary hypersurface $N_{\lambda\nu}$, there is a
$\TT^{\row(\lambda)}$-action, by the inductive hypothesis that
$\Mon_{\lambda}$ has a smooth $\TT^{\row(\lambda)}$ action.  As in the
previous discussion of the Gibbons--Manton bundles, the inclusion
$D_\nu \subset D_\lambda$ gives an inclusion of the corresponding tori
$\TT^{\row(\nu)} \hookrightarrow \TT^{\row(\lambda)}$ and this
inclusion gives the action of $\TT^{\row(\nu)}$ on $N_{\lambda\nu}$.

In order to make $\Sym_\nu$ act on the set of boundary
hypersurfaces, we must pick a lift $\sigma \to \Sigma_{\nu}$.  This
leads to a well-defined action on $\Mon_{\nu}$ with the claimed
properties because the choice of lift is compensated for by dividing
out by the groups $\Sym^0_{\lambda\nu}$.

We note that Theorem~\ref{cpct.conj} contains as a special case the
compactification of $\cM_k$ as a manifold with corners:

\begin{thm} The moduli space $\cM_k$ of dimension $4k-3$ has a
  compactification $\Mon_k$, which is a compact manifold with
  corners and a natural IBF structure. The boundary hypersurfaces are
  indexed by the $\Sigma_k$-orbits of partitions of $\bk$. If $\lambda$ and $\mu$ are
  partitions, then the corresponding boundary hypersurfaces intersect
  if and only if $\lambda$ and $\mu$ are comparable up to the action
  of $\Sigma_k$.  The boundary hypersurface corresponding to $\lambda$
  is
$$
\cI_{[\lambda]} = \cI_{\lambda}/\Sym_\lambda
$$
where
$$ \cI_\lambda = \Mon_{\lambda}\times_{\TT^{\row(\lambda)}}
  \cT_{\lambda k} \to  \FB_{\lambda k}
$$
and $\cT_{\lambda k}$ is the Gibbons--Manton bundle of type $\lambda$.
\end{thm}

\subsection{Low-Charge Examples Revisited}\label{sec:ex.rev}

If $k=1$, there is only the trivial partition $(1)$ of $k$ and $\cM_1 \cong \TT$ is already a compact space.

For $k=2$, the only nontrivial partition of $k$ is the `free' partition $0 = (1,1)$. Our compactification of $\cM_2$ is as a manifold
with boundary $\pa \Mon_2 = \cI_0$, which fibers over $B_{0 2} = S(\ol{\RR^6}/\RR^3)/\Sigma_2 \cong \RR P^2$ with fiber $\cM_{(1,1)} = \cM_1\times \cM_1 \cong \TT^2$.
The quotient of $\cM_2$ by $\TT$ (which acts diagonally on the fibre $\TT^2$ in the above) is the well-known Atiyah-Hitchin manifold of dimension 4, and we recover the known fact that it may be compactified by adding a boundary hypersurface which is a circle fibration over $\RR P^2$ \cite{HHM}.

The case $k=3$ is more interesting. The non-trivial partitions $0=(1,1,1)$ and $a = (1,2)$ lead to two boundary hypersurfaces: $\cI_{0} \longrightarrow B_{0 3}$ with fibre $F_{0} = \cM_1 \times \cM_1 \times \cM_1 \cong \TT^3$ and $\cI_{a} \longrightarrow B_{a 3}$ with fibre a compactification of $\cM_1 \times \cM_2$ (which in this case is simply $\cM_1 \times \Mon_2$ since $\cM_1$ is already compact). There is a single codimension two boundary face $\cI_{0} \cap \cI_{a}$.

For $k=4$, the non-trivial partitions are $0=(1,1,1,1)$, $a = (1,1,2)$, $b =
(1,3)$ and $c = (2,2)$ and hence $\Mon_4$ has four boundary hypersurfaces.  The
boundary faces of codimension $\ell$ are enumerated by the $\Sigma_4$ orbits of
chains of partitions of the set $\set{1,2,3,4}$ of length $\ell$, which in this
case are equivalent to length $\ell$ chains of integer partitions; in other
words these codimension $\ell$ faces are just the $\ell$-fold intersections of the boundary
hypersurfaces $\cI_{\alpha}$ for  $\alpha \in \set{0,a,b,c}$.  Thus there are five corners of codimension $2$:
$\cI_0 \cap \cI_a$, $\cI_0 \cap \cI_b$, $\cI_0 \cap \cI_c$, $\cI_a \cap \cI_b$ and $\cI_a \cap \cI_c$ and two
corners of codimension $3$: $\cI_0 \cap \cI_a \cap \cI_b$ and $\cI_0 \cap \cI_a \cap \cI_c$.
There are no corners of higher codimension.

For $k \geq 5$, the intersections of the boundary hypersurfaces $\cI_{\ast}$
are no longer connected in general, since there is a distinction between the
$\Sigma_k$ orbits of chains of set partitions, which enumerate the boundary
faces of a given codimension, and the chains of integer partitions, which
correspond to intersections of the $\cI_{\ast}$. For example in $k = 5$ the
intersection $\cI_{a} \cap \cI_{b}$ of the two boundary hypersurfaces
$\cI_{a}$, $a = (1,1,1,2)$ and $\cI_{b}$, $b = (2,3)$ is disconnected:
among its components are the quotients by $\Sigma_5$ of $\cI_{\lambda} \cap
\cI_{\nu}$ and $\cI_{\lambda} \cap \cI_{\nu'}$, where $\lambda =
\set{\set{1}, \set{2}, \set{3}, \set{4,5}}$, $\nu =
\set{\set{1,2,3},\set{4,5}}$ and $\nu' = \set{\set{1,2},\set{3,4,5}}$.

\subsection{Asymptotic Metrics} \label{S:model_metrics}

Now that we have described the boundary hypersurfaces of our
compactification in more detail, we can also give more information
about the metric.  We have already stated that it is a $\Phi$ metric
adapated to the IBF structure of the compactification $\Mon_k$.  But
there is also a relatively simple description in terms of the adapted
covers constructed in Proposition~\ref{thm:ibf.cover}.  Let
$W_\lambda$ be the set corresponding to the boundary hypersurface
labelled by the ($\Sigma_k$-orbit of) $\lambda$.

Denote by $T_{\scat,\lambda}$ the lift to $\MB_{\lambda,k}$ of the scattering
tangent bundle of $\ol{D_\lambda/D_k}$.  Denote by $\eta_\lambda$ the lift to
$\MB_{\lambda,k}$ of the euclidean metric on $\ol{D_\lambda/D_k}$, so
  that $\eta_{\lambda}$ is a smooth metric on $T_{\scat,\lambda}$.
Denote by $g_\lambda$ the riemannian
product metric on $\cM_\lambda$.
It is clear that these metrics descend to the quotients of these
spaces by $\Sym_\lambda$.

\begin{thm}
The boundary fibration $\phi_\lambda: N_\lambda \cap W_\lambda \to
\FB_\lambda/\Sym_{\lambda}$ admits a smooth extension $W_\lambda \to
U_\lambda$, a product neighbourhood of $\FB_\lambda$ in $\MB_\lambda$,
such that
\begin{itemize}
\item $T_\Phi$ is isomorphic to $\phi^*_\lambda T_{\scat,\lambda}
  \oplus T\cM_\lambda$ and
\item relative to this decomposition, $g_k$ is smooth and its
  restriction to $N_\lambda \cap W_\lambda$ is the direct sum
  $\eta_\lambda \oplus  g_\lambda$.
\end{itemize}
\end{thm}

\section{Decomposable Monopoles and Clusters}\label{S:cluster}
Throughout, we follow the observation that asymptotically, monopoles decompose into clusters of lower charge monopoles, cf.\ Sections \ref{sec.intro.metr.cpct} and \ref{sec.div.sqn}. For sequences of monopoles, this has been shown by Atiyah and Hitchin \cite[Prop. 3.8]{AH}.
As we wish to attach, in a consistent fashion (see Section \ref{sec.main.thm}), a collection of boundary hypersurfaces associated to such clusters to $\cM_k$, we need to identify not only limits of sequences in $\cM_k$ but \emph{asymptotic regions} that can be associated to proper clusters. And we need to show that these regions cover $\cM_k$ up to a relatively compact subset.

Given $k>1$, a type $a$ of $k$ and two parameters $R$, $\ve > 0$, the first thought to be large and the second small, we will define open sets $\cA_a(R,\ve) \subset \cM_k$ of decomposable monopoles. $R$ is a large radius of balls which we use to 'cover the components of the clusters` while $\ve$ is the reciprocal of a separation parameter giving a lower bound on the separation of these components. We then show that finitely many of such sets already suffice to cover $\cM_k$ and use this alongside Theorem \ref{cpct.conj} to show that any sequence in $\cM_k$ has a subsequence that either converges in $\cM_k$ or to an \emph{ideal configuration} in one of the boundary hypersurfaces $\cI_a$, proving that our ansatz indeed yields a compactification of $\cM_k$.

In the following, a pair $(A,\Phi)$ will always refer to framed monopole data, $(A,\Phi) \in \cC_{\Fr}$, solving the Bogomolny equation \eqref{e1.30.8.18}. By abuse of notation, we will call such a pair a (magnetic) monopoles as well.

\subsection{Decomposable Monopoles} We start by reviewing some key results from \cite{taubesminmax}. There, it is shown that there are numbers $\kappa(k)>0$, $N(k) \in \bN$, $R(k)>0$ and $c_0(k)>0$ with the following significance: For any charge $k$ monopole $(A,\Phi)$, we let
 \begin{equation}
  \widehat{\mathcal{U}}_{(A,\Phi)} = \Big\{\, z \in \RR^3 \,\Big|\, \int_{\mathbb{B}(z,1)} \big|F_A(z^\prime)\big|^2 dz^\prime > \tfrac{1}{2}\kappa(k)\,\Big\}
 \end{equation}
and define the \emph{strong-field region} of $(A,\Phi)$ by
 \begin{equation}\label{eq:strong.field}
  \mathcal{U}_{(A,\Phi)} = \Big\{\, z \in \RR^3 \,\Big|\, \dist\big(z,\widehat{\mathcal{U}}_{(A,\Phi)}\big) < 1 \,\Big\} \;.
 \end{equation}
Then, the set $\mathcal{U}_{(A,\Phi)}$ has $N$ connected components, where $0 < N \le N(k)$, and the centres of mass $\zeta_1,\dotsc,\zeta_N \in \RR^3$ of the connected components are uniformly bounded. We define the \emph{weak field region} to be the complement of the $R(k)$-neighbourhood of the strong-field region,
 \begin{equation}\label{eq:weak.field}
  \mathcal{W}_{(A,\Phi)} = \Big\{\, z \in \RR^3 \,\Big|\, \dist\big(z,\mathcal{U}_{(A,\Phi)}\big) > R(k) \,\Big\} \;.
 \end{equation}
On $\mathcal{W}_{(A,\Phi)}$, Taubes proved the following estimates \cite[C.1.4]{taubesminmax}:
 \begin{equation}\label{eq:taubes.c.1.4}
  1-\big|\Phi(z)\big| < \tfrac{1}{10} \quad,\qquad
  \big| \nabla_A \Phi (z) \big| < \tfrac{1}{10} \quad,\qquad
  \big|\nabla_A \nabla_A \Phi(z) \big| < \tfrac{1}{10} \;.
 \end{equation}
Moreover, he proves multipole estimates for both $\Phi$ and $\nabla_A\Phi$ \cite[C.2.1, C.3.1]{taubesminmax} on all of $\RR^3$: There exist numbers $\alpha_1, \dotsc, \alpha_N \in \RR$ depending on $(A,\Phi)$ and a constant $c_0(k)$ depending on $k$ only, so that
 \begin{equation}\label{eq:multipole}\begin{aligned}
  \left| \big|\Phi(z)\big| - 1 + \sum_{j=1}^N \frac{\alpha_j}{|z-\zeta_j|} \right|
   &\le c_0(k) \sum_{j=1}^N \frac{1}{|z-\zeta_j|^2} \\
  \left| \big(\Phi(z),\nabla_A\Phi(z)\big) + \sum_{j=1}^N \alpha_j d|z-\zeta_j|^{-1} \right|
   &\le c_0(k) \sum_{j=1}^N \frac{1}{|z-\zeta_j|^3} \\
  \left| \big[ \Phi(z),\nabla_A \Phi(z)\big] \right|
   &\le c_0(k) \sum_{j=1}^N e^{-\frac{1}{2}|z-\zeta_j|} \;.
 \end{aligned}\end{equation}
Let $W$ be the convex hull of a connected component of $\RR^3\setminus \mathcal{W}_{(A,\Phi)}$. As $|\Phi|$ is non-zero on $\partial W$, we can consider the mapping degree of $\Phi\,|\Phi|^{-1}$ over $\partial W$. This will of course be an integer, but not necessarily positive. Positivity is crucial, though, as we need to make sure that the individual components of a cluster have positive charge. A sufficient condition for positivity is obtained by comparing the size of $W$ to its distance to other non-weak regions:

\begin{lem}\label{thm:pos.charge} There is $R^\prime(k) \ge R(k) > 0$ such that for $R > R^\prime(k)$ the following holds: If there is $p_0 \in \RR^3$ satisfying
 \begin{equation}\label{condition}
  \dist\Big( \partial \mathbb{B}(p_0,R), \mathcal{U}_{(A,\Phi)} \Big) > R^{3/4} \quad\text{ and }\quad
  \mathbb{B}(p_0,R) \cap \mathcal{U}_{(A,\Phi)} \neq \emptyset \;,
 \end{equation}
then the set $\mathbb{B}(p_0,R) \cap \mathcal{U}_{(A,\Phi)}$ has strictly positive topological charge for $\Phi$, i.e.\
 \begin{equation}\label{pos.deg}
  \deg \frac{\Phi}{|\Phi|}\Big|_{\partial \mathbb{B}(p_0,R)} > 0 \;.
 \end{equation}
\end{lem}

\begin{proof} First of all, we need to do a small calculation: If $V$ is an open set having the property that $|\Phi|>0$ near $\partial V$, we may diagonalise $\Phi$ near $\partial V$ and obtain
 \begin{equation}\label{calculation}\begin{aligned}
    \int_V \big|F_A(z)\big|^2dz
     &= \tfrac{1}{2} \int_V \tr \big(F_A \wedge \ast F_A\big)
      = \tfrac{1}{2} \int_V d_A \tr \big(F_A \wedge \Phi\big) \\
     &= \tfrac{1}{2} \int_{\partial V} \tr\big(F_A \wedge \Phi\big)
      = \int_{\partial V} i F_a \varphi \\
     &= 2\pi c_1\big(L|_{\partial V}\big) - i\int_{\partial V} F_a (1-|\Phi|) \;,
 \end{aligned}\end{equation}
where, near $\partial V$, $\Phi = \mathrm{diag}(i\varphi,-i\varphi)$ and
 \[ F_A = \mathrm{diag}(F_a,-F_a) + \{\text{\emph{off-diagonal terms}}\} \;, \]
both with respect to the bundle decomposition $\mathrm{ad}E = L \oplus L^{-1}$ into eigenbundles of $\Phi$, and where $c_1(L)$ denotes the Chern number of $L$. Now suppose \eqref{condition} is satisfied for some $p_0 \in \RR^3$. Then, combining \eqref{calculation} with \eqref{eq:strong.field}, we get
 \[ \kappa(k) < 4 \pi \, c_1\big(L|_{\partial \mathbb{B}(p_0,R)}\big) - 2i\int_{\partial \mathbb{B}(p_0,R)} F_a\big(1-|\Phi|\big) \;. \]
The latter term is real and can be bounded using \eqref{eq:multipole}:
 \begin{equation}\begin{aligned}
  \left| 2 i\int_{\partial \mathbb{B}(p_0,R)} F_a\big(1-|\Phi|\big) \right|
     &\le 2\big|\partial \mathbb{B}(p_0,R)\big| \max \Big( \big|F_A\big| \, \big|1-|\Phi|\big|\Big) \\
     &\le 8\pi R^2 \,\cdot\,c_0(k)\, R^{-9/4} = c \, R^{-1/4} \;,
 \end{aligned}\end{equation}
where $c$ depends on $k$ only. Thus, we arrive at
 \[ 4\pi \, c_1\big(L|_{\partial \mathbb{B}(p_0,R)}\big) > \kappa(k) - c\,R^{-1/4} \;. \]
Then, if we choose $R^\prime(k) = \max\big\{R(k),\big(\tfrac{c}{\kappa(k)}\big)^4\big\}$ and $R \ge R^\prime(k)$, we obtain $c_1(L|_{\partial \mathbb{B}(p_0,R)}) > 0$ and consequently \eqref{pos.deg}.
\end{proof}

Thus, whenever we can put some number of connected components of the strong-field region in a ball whose boundary is sufficiently far away from the strong-field region, this ball contains positive charge. In particular, any collection of components of the strong-field region which is widely separated from the rest carries positive charge.

\begin{dfn}\label{def:decomposables} The set of \emph{decomposable monopoles} of type $a=(k_1,\dotsc,k_n)$, size $R>0$ and separation $\ve>0$ is the set $A_a(R,\ve)$ of pairs $(A,\Phi)$ as above for which there are $p_1,\dotsc,p_n \in \RR^3$, so that:
 \begin{enumerate}
  \item $\RR^3 \setminus \bigcup_{j=1}^l \mathbb{B}(p_j,R) \subset \mathcal{W}_{(A,\Phi)}$
  \item $\mathrm{deg} \frac{\Phi}{|\Phi|}\big|_{\partial \mathbb{B}(p_j,R)} = 2\pi k_j$
  \item $\sum_{i < j} |p_i-p_j|^{-1} < \ve$
 \end{enumerate}
\end{dfn}

The definition of the weak-field region $\mathcal{W}_{(A,\Phi)}$ lifts to $\cM_k$, as does the definition of $A_a(R,\ve)$, resulting in sets of centered monopoles $\cA_a(R,\ve)$, which we call decomposable monopoles as well. Notice that if $a = k$ is the trivial type, $\cA_k(R,\ve)$ does not depend on the separation parameter $\ve$ (as there is but a single $p_j$), we will sometimes denote this set by $\cA_k(R)$.

\begin{prop}\label{thm:decomposables.cover} For each $k \in \bN$, there are $M(k) \in \bN$, $\ve(k)>0$ and radii $R_0(k)$, $\dotsc$, $R_{M(k)}(k)$, $R_j(k) > R^\prime(k)$ with $R^\prime(k)$ as in Lemma \ref{thm:pos.charge}, so that
 \begin{equation}\label{eq:mon.cover}
  \cM_k \subset \cA_k(R_0(k)) \cup
   \bigcup_{a} \bigcup_{j=1}^{M(k)} \cA_a\big(R_j(k),\ve(k)\big) \;,
 \end{equation}
where $a$ runs over all proper types of $k$.
\end{prop}

\begin{proof} The proof is constructive, to each monopole we associate one of the sets $\cA_a\big(R_j(k),\ve(k)\big)$. The principle is as follows: Given a monopole $m$, take its strong field region $\mathcal{U}_m$ and cover it by balls, checking whether the boundaries of these are sufficiently far away from the strong-field region. If not, consecutively enlarging the balls and grouping together more connected components of $\mathcal{U}_m$, we show that in the end we arrive at a cover as in Definition \ref{def:decomposables}. Hereby, the estimates from Taubes and Proposition \ref{thm:pos.charge} will be essential. Then, we argue that a finite set of radii and one single $\ve$ is enough to achieve this.

So let $(A,\Phi)$ be a charge $k$ monopole and $\mathcal{U}_{(A,\Phi)}$ be its strong-field region. There are numbers $N=N(k)$ and $d=d(k)$ so that $\mathcal{U}_{(A,\Phi)}$ has at most $N$ connected components, say $U_1,\dotsc,U_N$, and each of them has diameter bounded above by $d$. (This follows readily from the definition of the strong-field region and the fact that the total curvature of $A$ is $2\pi k$.) We will now describe an algorithm which gives centres and balls as in Definition \ref{def:decomposables}. For definiteness, let us be detailed here. We will denote the different steps in the algorithm by a letter $t$ (used as superscripts). Moreover, the letter $\omega$ will denote auxiliary partitions not directly related to the partitions $\lambda$ or the types $a$.

 \begin{enumerate}\setcounter{enumi}{-1}
  \item Define the following data:
         \begin{align*}
          \omega^0 &= \Big\{\, \{1\},\dotsc,\{N\}\,\Big\} \quad,&
          d^0 &= \max\{d,15\} \quad, \\
          R^0 &= d^0 + R^\prime(k) + 1 \quad,&
          \gamma^0 &= 3R^0 - \tfrac{1}{2}d^0 \quad,\\
          p_j^0 &= \big|U_j\big|^{-1}\int_{U_j} z\, dz \quad,&
          \mathbb{B}_j^0 &= \mathbb{B}(p_j^0,R^0)
         \end{align*}
      and set $t=0$.
  \item Partition the set $I^t = \{1,\dotsc,\row(\omega^t)\}$ as follows: For each $j \in I^t$, let
         \[ I_j^t = \{\, j^\prime \,:\, \big|p_j^t - p_{j^\prime}^t\big| \le \gamma^t\,\} \]
      and then let $\omega^{t+1}$ be the finest partition of $I^t$ so that each $I_j^t$ is contained in a single block of $\omega^{t+1}$.
  \item If $\omega^{t+1}$ consists of singletons only, stop. Else continue with step (3).
  \item Define the next batch of data:
         \begin{align*}
          d^{t+1} &= \Big(\max_{1 \le j \le \row(\omega^{t+1})} |\omega_j^{t+1}|
                          - 1\Big)\gamma^t + 2R^t \quad,&&  \\
          R^{t+1} &= d^{t+1} + 1 \quad,&
          \gamma^{t+1} &= 3R^{t+1} - \tfrac{1}{2}d^{t+1} \quad,\\
          p_j^{t+1} &= |\omega_j^{t+1}|^{-1}\sum_{j^\prime \in \omega_j^{t+1}}p_{j^\prime}^t \quad,&
          \mathbb{B}_j^{t+1} &= \mathbb{B}(p_j^{t+1},R^{t+1}) \quad,
         \end{align*}
      do $t \mapsto t+1$ and repeat from step (1).
 \end{enumerate}
Let us look closer at this algorithm. Step (0) is initialisation of data. Here, blocks of $\omega^0$ correspond to the connected components $U_1,\dotsc,U_N$ whose separation we need to check, $R^0$ is already chosen sufficiently large so that we can apply Lemma \ref{thm:pos.charge}, $\gamma^0$ is a separation threshold against which we check and the remaining data is either self-explanatory or auxiliary. In step (1), members of the same block in $\omega^t$ correspond to sets of connected components of $\mathcal{U}_{(A,\Phi)}$ that cannot be widely separated in the sense of Lemma \ref{thm:pos.charge}, while the distance between members of different blocks is sufficiently large. Hence, if $\omega^t$ consists of singletons only, this corresponds to a widely separated cluster, and we can halt in step (2). In step (3), we enlarge the radii (so as to be able to put all components that were not sufficiently far away from each other into single balls) and correspondingly increase the separation threshold $\gamma^t$.

In each run, we either obtain $\row(\omega^{t+1}) < \row(\omega^{t})$ or else the algorithm stops in step (2). Since $\row(\omega^0) = N \le N(k)$ is uniformly bounded, after at most $N(k)-1$ runs, the algorithm stops. Say it stops after ${t_0}$ runs and we end up with data $d^{t_0}$, $R^{t_0}$, $\gamma^{t_0}$, $p_1^{t_0},\dotsc,p_n^{t_0}$, $\mathbb{B}_1^{t_0},\dotsc,\mathbb{B}_n^{t_0}$.
By choice of $d^0$, $p_j^0$ and $R^0$, the union of the balls $\mathbb{B}_j^0$ covers the $R(k)$-neighbourhood of $\mathcal{U}_{(A,\Phi)}$ and consequently we have $\RR^3 \setminus \bigcup_j \mathbb{B}_j^0 \subset \mathcal{W}_{(A,\Phi)}$. Moreover, the construction of $\omega^{t+1}$ and the choice of $d^{t+1}$ and $R^{t+1}$ ensure that for each $j^\prime$ there is $j$ with $\mathbb{B}_{j^\prime}^{t} \subset \mathbb{B}_j^{t+1}$. Thus,
 \[ \RR^3 \setminus \bigcup_{j=1}^n \mathbb{B}_j^{t_0} \subset \mathcal{W}_{(A,\Phi)} \;. \]
As we stop if and only if $\omega^{{t_0}+1}$ consists of singletons only, we have $|p_j^{t_0} - p_i^{t_0}| > \gamma^{t_0}$ for all $1 \le i < j \le n$ and since $\gamma^{t_0} - 2R^{t_0} = R^{t_0} - \tfrac{1}{2}d^{t_0} = \tfrac{1}{2}d^{t_0} + 1 > 0$, the balls $\mathbb{B}_j^{t_0}$ are mutually disjoint.

Now, as $R^0 > 16$ and $R^{t+1} \ge R^t$, we have $1-\big(R^{t_0}\big)^{-1/4} > \tfrac{1}{2}$ and jointly with $R^{t_0} > d^{t_0}$ this implies $R^{t_0} - \tfrac{1}{2}d^{t_0} > \big(R^{t_0}\big)^{3/4}$. Then,
 \begin{equation}\begin{aligned}
  \dist\Big(\partial \mathbb{B}_j^{t_0}, \mathcal{U}_{(A,\Phi)} \cap \mathbb{B}_j^{t_0}\Big)
  &> \dist\Big(\partial \mathbb{B}_j^{t_0}, \mathbb{B}(p_j^{t_0},\tfrac{1}{2}d^{t_0})\Big) \\
  &= R^{t_0} - \tfrac{1}{2}d^{t_0} > \big(R^{t_0}\big)^{3/4}
 \end{aligned}\end{equation}
and, for $i \neq j$,
 \begin{equation}\begin{aligned}
  \dist\Big(\partial \mathbb{B}_j^{t_0}, \mathcal{U}_{(A,\Phi)} \cap \mathbb{B}_i^{t_0}\Big)
  &> \dist\Big(\partial \mathbb{B}_j^{t_0}, \partial \mathbb{B}_i^{t_0}\Big) \\
  &> \gamma^{t_0} - 2 R^{t_0} = R^{t_0} - \tfrac{1}{2}d^{t_0} > \big(R^{t_0}\big)^{3/4} \;.
 \end{aligned}\end{equation}
Which shows that $\dist\big(\mathbb{B}_j^{t_0},\mathcal{U}_{(A,\Phi)}\big) > \big(R^{t_0}\big)^{3/4}$ and by Lemma \ref{thm:pos.charge} that the Higgs-field has positive degree $2\pi k_j$ around each $\mathbb{B}_j^{t_0}$, for some $1 \le k_j \le k$.
As the complement of the union of the balls is contained in the weak-field region, these add up to the total charge $k$ and we obtain a type $a=(k_1,\dotsc,k_n)$ of $k$ of length $n$.
Finally, we note that $\big|p_j^{t_0} - p_i^{t_0}\big| > \gamma^{t_0}$ implies $\sum_{i < j} \big|p_j^{t_0} - p_i^{t_0}\big|^{-1} < n^2 (\gamma^{t_0})^{-1}$ and so we have $(A,\Phi) \in A_a\big(R^{t_0},n^2 (\gamma^{t_0})^{-1}\big)$.

Apart from the centres $p_j^0$, the initial data is independent of the choice of $(A,\Phi)$ and in fact depends on $k$, only. ($\omega^0$ depends on $(A,\Phi)$, but since $N \le N(k)$ is uniformly bounded, this can be remedied by repeating one of the connected components $U_j$ sufficiently many times.) Looking at step (3) which defines the next batch of data, we see that only a finite number $M(k)$ of different radii may arise, depending on the maximal length of the blocks of the $\omega^{t}$. These are the radii $R_1(k),\dotsc, R_{M(k)}(k)$. $R_0(k)$ can be chosen as the maximum, $R_0(k) = \max_{1 \le j \le M(k)} R_j(k)$, since clearly $A_k(R,\ve) \subset A_k(R^\prime,\ve)$ for $R < R^\prime$.
Similarly, we can only encounter a finite number of $\gamma^{t_0}$'s and may take their minimum, say $\gamma$. (This is in fact $\gamma^0$ since $\gamma^t \le \gamma^{t+1}$.)
As we also have the trivial inclusion $A_a(R,\ve^\prime) \subset A_a(R,\ve)$ whenever $\ve^\prime < \ve$, we can choose $\ve(k) = \frac{k^2}{\gamma}$ to obtain sets $A_k(R_0(k))$ and $A_a(R_j(k),\ve(k))$ whose union contains all possible $(A,\Phi)$.
Factoring out by translations and the framed gauge group, we arrive at the claim.
\end{proof}

Given a charge $k$ monopole $(A,\Phi)$, we call the data $a$, $p_1,\dotsc,p_n$ and $R$ obtained in this construction the \emph{decomposition data}, i.e., $p_j = p_j^M$ and $R = R_M$. Since types are unordered, we need to consider the configuration $p_1,\dotsc,p_n$ as being unordered as well.

\begin{rmk} Notice that by definition of the sets $\cA_a(R,\ve)$, there is a map
 \[ \cA_a(R,\ve) \longrightarrow (0,\ve) \times B^\prime_{\lambda k} \;, \]
mapping a monopole to the unordered configuration $p_1,\dotsc,p_n$:  The left hand side of item (3) of Definition \ref{def:decomposables} defines a boundary defining function $\rho^\prime_{\lambda k}$ for $B^\prime_{\lambda k}$. Part of Theorem \ref{cpct.conj} is then that, at infinity, this map yields a fibration with base $B^\prime_{\lambda k}$. What we have shown so far is that there is a set $K = \cA_0(R_0(k)) \subset \cM_k$ and a well-defined map
 \begin{equation}\label{eq:cluster.config}
  \cM_k \setminus K \longrightarrow \bigcup_{a}\, (0,\ve) \times B_a \;,
 \end{equation}
associating a `cluster configuration' (of a proper type) to each monopole outside of the core region. In the next section, we will see that this core region $K$ is in fact relatively compact.
\end{rmk}

\subsection{Limits of Decomposable Sequences}\label{S:limits.decomp} Given any sequence $m^{s^\prime} \in \cM_k$, it is clear that there is a subsequence $m^s$ and a set $\cA_a(R,\ve) = \cA_a(R_j(k),\ve(k))$ so that $m^s \in \cA_a(R,\ve)$ for all $s \in \bN$. Let the respective decomposition data be $a$, $p_1^s,\dotsc,p_n^s$, $R$ and define a sequence in $\RR_+$ by
 \begin{equation}\label{eq:base.sqn}
  \ve^s := \rho^\prime_{\lambda k}(p_1^s,\dotsc,p_n^s) = \sum_{i<j} \big|p_i^s-p_j^s\big|^{-1} < \ve \;.
 \end{equation}
Comparing the results of \cite[Prop 3.8]{AH} and of Proposition \ref{thm:decomposables.cover}, it is not difficult to see that they necessarily lead to the same type $a$ of $k$ and to sequences of bounded distances: $\big|p_j^s - w_j^s\big| \le c$ for all $s \in \bN$, where $w_j^s$ denotes the set of sequences obtained from \cite[Prop. 3.8]{AH}. The reason for this is that the sequences $w_j^s$ of \cite{AH} arise as sequences of zeroes of the Higgs-fields and that Taubes' estimates \eqref{eq:taubes.c.1.4} show that a zero of the Higgs-field is contained in the complement of the weak-field region and thus, in our case, in one of the balls $\mathbb{B}(p_j^s,R)$. The same line of thought, applied to the case of $a = (k)$, yields the following:

 \begin{lem}\label{thm:bdd.Higgs} Let $m^s$ be a sequence of charge $k$-monopoles so that $\big(\RR^3\setminus \mathbb{B}(p,R) \big) \subset \cW_{m^s}$ for all $s \in \bN$. Then, $m^s$ has a convergent subsequence. In particular, the sets $\cA_k(R_0)$ are relatively compact.
 \end{lem}

Returning to our original sequence, we may use this to show the following dichotomy:

 \begin{prop}\label{thm:convergence.alternatives} Any sequence in $\cM_k$ has a subsequence $m^s$ so that either
 \begin{enumerate}
  \item there is $m \in \cM_k$ so that $m^s \longrightarrow m$ uniformly on compact subsets, or
  \item there is a set $\cA_a(R,\ve)$ so that $m^s \in \cA_a(R,\ve)$ for all $s \in \bN$ and $\ve^s \longrightarrow 0$, where $\ve^s$ is defined as in \eqref{eq:base.sqn}.
 \end{enumerate}
 In particular, if $m^s \in \cM_k$ is a sequence leaving any compact subset, then $\ve^s \longrightarrow 0$.
 \end{prop}

\begin{proof} Suppose there is $c>0$ so that $\ve^s \ge c$ for all $s \in \bN$. Then, $|p_i^s - p_j^s| \le c^{-1}(1+\ve)$ is uniformly bounded and we can find a single center of mass $\tilde{p} \in \RR^3$ and a single radius $\wt R > R_0$ so that $\mathbb{B}(p_j^s,R) \subset \mathbb{B}(\tilde{p},\wt R)$ for all $1 \le j \le n$. Lemma \ref{thm:bdd.Higgs} then shows that there is a subsequence converging in $\cM_k$. Thus, one and only one of items (1) or (2) holds.

This also shows that bounding $\rho^\prime_{\lambda k}$ from below defines a relatively compact subset in $\cM_k$. Whence if we have a sequence that leaves any compact subset, we necessarily have a subsequence on which $\ve^s = \rho^\prime_{\lambda k}(p_1^s,\dotsc,p_n^s) \longrightarrow 0$.
\end{proof}

Thus, the `non-compactness' of $\cM_k$ is completely described by families of monopoles for which $\ve^s \longrightarrow 0$. But these are precisely the families approaching one of the boundary hypersurfaces from Theorem \ref{cpct.conj}.

\begin{thm}\label{thm:mon.cpct} $\Mon_k$ is compact.
\end{thm}

\begin{proof} Since its boundary hypersurfaces are compact, it clearly suffices to show that every sequence in $\cM_k$ has a subsequence that converges in $\Mon_k$. By Proposition \ref{thm:convergence.alternatives}, on a subsequence $m^s$ we either have convergence in $\cM_k$ or we have $\ve^s \longrightarrow 0$. So let us assume the latter is true and that we are in the situation of item (2) of Proposition \ref{thm:convergence.alternatives}.

Since $\ve^s$ is the value of the boundary defining function $\rho^\prime_{\lambda k}$ for $B^\prime_{\lambda k}$ in its respective many-body compactification, we see that the sequence $(p_1^s,\dotsc,p_n^s)$ converges to an element $\xi \in B^\prime_{\lambda k}$. Moreover, there are $m_j \in \cM_{k_j}$ so that the translates of $m^s$ by $p_j^s$ converge to $m_j$: $T_j^*m^s \longrightarrow m_j$ uniformly on compact subsets.
But since the ideal data $(\xi, m_1,\dotsc,m_n)$ defines an element of $\cI_a$, (2) implies that there is a neighbourhood $U_a$ of $\cI_a$ and $s_0 \in \bN$ so that, for all $s \ge s_0$, we have $m^s \in U_a$. As any such neighbourhood is relatively compact in $\Mon_k$, there is a further subsequence that converges in $\ol{U}_a \subset \Mon_k$.
\end{proof}

Let us compare this to earlier results: In Theorem \ref{cpct.conj}, we claim that $\Mon_k$ carries an iterated boundary fibration, cf.\ Section \ref{sec.main.thm}. Moreover, we have shown in Proposition \ref{thm:ibf.cover} that for any compact MWC with IBF, there is a cover $\{W_\lambda\}$ of a neighbourhood of its boundary so that $W_\lambda \cap W_\mu = \emptyset$ if $\lambda$ and $\mu$ are not comparable and so that the boundary fibrations $\phi_\lambda : N_\lambda \longrightarrow B_\lambda$ restrict to be surjective with fibres contained in a relatively compact subset of the interior of the unrestricted fibres.

Since the sets $\cA_a(R_j,\ve)$ cover $\cM_k$ up to a relatively compact subset, their closures in $\Mon_k$ clearly cover a neighbourhood of $\partial \Mon_k$. Looking at the definition of the sets $\cA_a(R,\ve)$ it is clear that an intersection $\cA_a(R,\ve) \cap \cA_b(R^\prime,\ve)$ is non-empty if and only if the types $a$ and $b$ are comparable: Assume $R \le R^\prime$ and that the intersection is not empty. Since there are sets of balls $\{\mathbb{B}(p_j,R)\}$ and $\{\mathbb{B}(p_j^\prime,R^\prime)\}$ covering the strong-field region of the same monopole, by looking at intersections we obtain a surjective map $\{p_j\} \longrightarrow \{p_j^\prime\}$ and in this way $a \le b$.

For the moment, we will not show that the restrictions of the boundary fibrations are still surjective (this will be shown in Part II), but consider the fibres instead. If $m^s \in \cA_a(R,\ve)$ for all $s$, by not allowing the radius $R$ of the balls $\mathbb{B}(p_j^s,R)$ to grow along the sequence, we exclude the possibility of the monopole `falling apart further'. Asymptotically speaking, in terms of fibres $\Mon_{\lambda k}$ and bases $B^\prime_{\lambda k}$, the sequence ends up in and stays in the set
 \[ \cA_{k_1}(R) \times \dotsm \times \cA_{k_n}(R) \subset \cM_{k_1} \times \dotsm \times \cM_{k_n} \subset \Mon_{\lambda k} \;, \]
which by Lemma \ref{thm:bdd.Higgs} and Theorem \ref{cpct.conj} is a relatively compact subset of the interior of the fibre $\Mon_{\lambda k}$. Thus, the cover \eqref{eq:mon.cover} is a cover of the type constructed in Proposition \ref{thm:ibf.cover}.


\section{Proof of the Sen Conjecture, Coprime Case}\label{psen}

Let us now use Theorem \ref{cpct.conj}, Proposition \ref{prop.mrs.bounded.cover} and results from \cite{SS} to prove the coprime case of the Sen Conjecture.

As in Section \ref{sec.background}, let $\unicov{k}$ denote the space
of strongly centred monopoles of charge $k$, which is the universal
cover of the hyperK\"ahler quotient $\cMo_k$.
The latter space has fundamental group isomorphic to $\ZZ_k$ and thus the deck
transformations of $\unicov{k}$ are given by the subgroup $\ZZ_k
\subset \TT$.

Before proceeding we make two further remarks about the relation
between the compactification $\Mon_{k}$ and the universal cover. The
first observation is that for any compact manifold with corners $M$,
with interior $M^0 = M\setminus \p M$, it is the case that $M$ and
$M^0$ have the same homotopy type.  This can be seen by choosing a
global boundary defining function $\rho$, say.  If
$$
M_\delta = M \setminus \{\rho < \delta\}
$$
it is easy to see that $M_\delta$ is homotopy equivalent to both $M$
and $M^0$.  In particular, the homotopy types of $\Mon_k/\TT$ and
$\cM_k/\TT$ are the same and both have fundamental group $\ZZ_k$.

The next observation is that if $M$ is a compact MWC with finite
fundamental group, then its universal cover $\wt{M}$ is, in a natural
way, again a compact MWC. The proof is an adaptation of the proof that
the universal cover of a smooth manifold is in a natural way again
smooth.

Combining these two points, we see that the compactification $\Mon_k$
of $\cM_k$ yields a compactification of $\unicov{k}$, which we will call
$\cunicov{k}$ in this section, and the IBF and metric properties proved for $\Mon_k$
hold also, simply by lifting, for the compactification $\cunicov{k}$.

If $\cH^i_k$ denotes the space of $L^2$ harmonic forms of degree $i$ on $\unicov{k}$, we can decompose $\cH^i_k$ according to the $\ZZ_k$-action and denote by $\cH^i_{k,\ell}$ the component in $\cH^i_k$ on which $\zeta \in \ZZ_k$ acts by multiplication with $\zeta^\ell$.
Analogously, we write $H^*(U)_\ell$, $H_c^*(U)_\ell$ and $H_2^*(U)_\ell$ for the respective components of the de Rham cohomology, the de Rham cohomology with compact supports and the $L^2$-cohomology of $U \subset \unicov{k}$. (The latter meaning the subcomplex of the de Rham complex consisting of smooth forms $\alpha$ for which both $\alpha$ and $\rd\alpha$ are square-integrable on $U$.) The integer $\ell$ is also referred to as the electric charge.

The Sen Conjecture can be stated as follows (cf.\ \cite{SS,sen1994strong}):
 \begin{enumerate}
  \item[(S.1)] If $k$ and $\ell$ are coprime, then $\cH^{2k-2}_{k,\ell} \cong \CC$ and $\cH^i_{k,\ell} = 0$ for $i \neq 2k-2$, and
  \item[(S.2)] if $k$ and $\ell$ have a common factor, then $\cH^i_{k,\ell} = 0$ for all $i$.
 \end{enumerate}
We will prove (S.1), the coprime case, by adapting an argument in \cite[\S2]{SS} and proving the following.

\begin{thm}\label{sen.1}
Let the integers $k$ and $\ell$ be coprime. Then the space $\cH^{i}_{k,\ell}$ of harmonic forms of degree $i$ and electric charge $\ell$ is canonically isomorphic to
 \begin{equation}\label{sen.image}
  \Im\big( H^i_c(\unicov{k})_\ell \longrightarrow H^i(\unicov{k})_\ell \big) \;.
 \end{equation}
In particular, the coprime case of the Sen Conjecture holds true:
 \begin{equation}\label{eq:sen.2}
  \cH_{k,\ell}^{i} \cong \begin{cases} \CC &\text{if $i = 2k-2$,} \\ 0 &\text{else.} \end{cases}
 \end{equation}
\end{thm}
Part of the proof will be showing the existence of a finite open cover $\{V_i\}$ of the compactification $\cunicov{k}$ of $\unicov{k}$ with the following properties:
 \begin{enumerate}
  \item $V_0$ is relatively compact
  \item For each $i>0$ there is a proper partition $\lambda_i$ of $\bf k$ so that $V_i$ can be identified with an open set of ordered clusters of monopoles of type $\lambda_i$
  \item $\TT^{\row(\lambda_i)}$ acts on $V_i$, extending the action of $\TT$ on $V_i \subset \cM_k$; this action is by near isometries and its orbits are of bounded size
  \item There is a partition of unity $\{\chi_i\}$ subordinate to the cover $\{V_i\}$ so that each $|\rd \chi_i|$ is bounded
 \end{enumerate}
Then, we will proceed as in \cite{SS} and use the $\TT^{\row(\lambda_i)}$-action on restrictions of the $V_i$ to $\unicov{k}$ in order to reduce $H_2^i(\unicov{k})_\ell$ to $H_c^i(\unicov{k})_\ell$. This will lead to a proof of Theorem \ref{sen.1}.

\begin{lem}\label{sen:cover} There exists a finite open cover $\{V_i\}$ of $\cunicov{k}$ satisfying conditions \emph{(1) -- (4)}.
\end{lem}

\begin{proof} We will first show the existence of a cover of $\Mon_k$ satisfying (1)--(3), then lift this cover to $\cunicov{k}$ and use Proposition \ref{prop.mrs.bounded.cover} to obtain (4).

By Theorem \ref{cpct.conj}, for each boundary hypersurface $N$ of $\Mon_k$ we have a $\Stab_{\Sigma_k}(D_k)$-orbit (i.e.\ a type) $[\lambda]$ of a partition $\lambda$ of $k$ and a fibration $\phi_\lambda : N = N_{\lambda k} \longrightarrow B^\prime_{\lambda k}$, where $B^\prime_{\lambda k}$ is the ideal or free boundary of the space of unordered configurations of type $[\lambda]$.
$B^\prime_{\lambda k}$ is the quotient $B_{\lambda k}/\Sym_{\lambda k}^0$ of the ideal boundary of the space of ordered configurations associated to $\lambda$, by the action of the symmetry group $\Sym_{\lambda k}^0 = \Sym_\lambda$, cf.\ \eqref{def.sigma.lambda}. Let $\pi_\lambda$ denote the canonic quotient map.
Since $B^\prime_{\lambda k}$ is compact, we can choose an open cover by a finite number of connected open sets, say $\wt\cO_{j,\lambda}$, with the property that $\pi_\lambda^{-1}(\wt\cO_{j,\lambda}) \subset B_{\lambda k}$ is diffeomorphic to the union of $|\Sym_{\lambda}|$ disjoint copies of $\wt\cO_{j,\lambda}$. For each $j$, we pick one of these lifts and denote it by $\cO_{j,\lambda}$.

Furthermore, since $N_{\lambda k}$ is compact, we can pick a finite open cover $\{\cU_{i,\lambda}\}$ of $N_{\lambda k}$ and, as the $\pi_\lambda(\cO_{j,\lambda})$ cover $B^\prime_{\lambda k}$, refine this cover in such a way that each $\phi_\lambda(\cU_{i,\lambda})$ is contained in one $\pi_\lambda(\cO_{j,\lambda})$. Then, the $\cU_{i,\lambda}$ constitute a finite open cover of $N_{\lambda k}$ and each of the elements of this cover can be identified with an open set of ordered ideal configurations associated to the partition $\lambda$.

Now choose a boundary product structure compatible to the IBF of
$\Mon_k$ as in Definition \ref{dfn:bps}. Using the retraction
$v_\lambda$ (cf.\ the paragraphs after Definition \ref{dfn:bps}), we may push
out the sets $\cU_{i,\lambda}$ to obtain a cover of a product neighbourhood
of $N_{\lambda k}$. Doing this for all proper types $[\lambda]$ of $k$, we obtain a cover of a product neighbourhood of $\partial \Mon_k$. Letting $\cU_0$ be an open neighbourhood of
 \[ \Mon_k \setminus \bigcup_{i,\lambda} \overline{\cU_{i,\lambda}} \;, \]
we obtain a cover of $\Mon_k$.

If $\wh \pi : \cunicov{k} \longrightarrow \unicov{k}$ denotes the quotient map, we may refine the cover $\{\cU_{i,\lambda}\}$ of $\Mon_k$ so that for each element of it, ${\wh \pi}^{-1}(\cU_{i,\lambda})$ consists of $k$ disjoint open sets. These sets form a cover $\{\wh \cU_{i,\lambda}\}$ of $\cunicov{k}$. Now let $\{W_i\}$ be a cover for $\cunicov{k}$ obtained by Proposition \ref{prop.mrs.bounded.cover} and take $\{V_i\}$ to be a common refinement of $\{W_i\}$ and $\{\wh \cU_{i,\lambda}\}$.

Conditions (1) and (2) are clearly satisfied by construction, as is condition (3): Each $V_i$ is identifiable with an open set of ordered clusters of monopoles of type $\lambda_i$, in particular $\TT^{\row(\lambda_i)}$ acts freely on $V_i$ extending the action of $\TT$ given on $\cM_k$. Moreover, $\TT^{\row(\lambda_i)}$ acts fibre-wise by isometries and, due to the form of the metric (cf.\ Proposition \ref{prop.mrs.bounded.cover} and Theorem \ref{cpct.conj}), acts by isometries with bounded orbits on $V_i$. Lastly, condition (4) follows since the cover $\{V_i\}$ is a refinement of the cover from Proposition \ref{prop.mrs.bounded.cover}.
\end{proof}

Now consider sets $\wt V_i = \big(V_i \setminus \partial\cunicov{k}\big) \subset \unicov{k}$. If the type corresponding to $V_i$ (downstairs in $\Mon_k$) is $a_i = [\lambda_i] = (k_1, \dotsc, k_n)$, then
 \[ G_{\lambda_i} = \big\{\, (\zeta_1,\dotsc,\zeta_n) \in \TT^n \,\big|\, \prod \zeta_j^{k_j} = 1 \,\big\} \]
acts on $\wt V_i$. The diagonal subgroup of $G_{\lambda_i}$ is isomorphic to $\ZZ_k$ and acts by deck transformations on $\wt V_i$ (cf.\ \cite[p. 779]{SS}).

\begin{lem} If $U$ is a $G_{\lambda_i}$-stable open submanifold of $\wt V_i$, where $i>0$, and $k$ and $\ell$ are coprime, then $H_2^*(U)_{\ell} = 0$.
\end{lem}

\begin{proof} The argument is an extension of the proof of \cite[Lemma 3.1]{SS}. If $d$ is the greatest common divisor of the numbers $k_1,\ldots, k_n$ in $a_i$, then the vector $(k_1,\ldots,k_n)$ is $d$ times a primitive vector in $\ZZ^n$. Thus we can find $(n-1)$ vectors which, along with $(k_1,\ldots,k_n)$ span a lattice of index $d$ in $\ZZ^n$. Thus $G_{\lambda_i} = \TT^{n-1} \times \ZZ_{d}$.
If the action of $(\zeta,\ldots,\zeta)$ on $\unicov{k}$ is denoted by $A_\zeta$, then it follows that the diffeomorphism $A_\zeta^d$ is in the identity component of $G_{\lambda_i}$ and hence linked to it by a homotopy generated by a bounded vector field.
Thus, if $\alpha$ represents an element of $H_2^*(U)_\ell$, we have $(A_\zeta^d)^*\alpha - \alpha = \rd\beta$ and from $(A_\zeta^d)^*\alpha = \zeta^{\ell d} \alpha$ and $\zeta^{\ell d} \neq 1$, we obtain $\alpha = \rd\big( (\zeta^{\ell d}-1)^{-1} \beta \big)$.
\end{proof}

\begin{lem} If $k$ and $\ell$ are coprime, the map $H_c^*(\unicov{k})_\ell \longrightarrow \cH^*_{k,\ell}$ given by orthogonal projection is onto.
\end{lem}

\begin{proof} Let $\{V_i\}$ denote the cover from Lemma \ref{sen:cover} and again write $\wt V_i = V_i \setminus \partial \cunicov{k}$. Then, $\wt V = \bigcup_{i>0} \wt V_i$ is a finite union and by condition (4), there is a smooth partition of unity $\{\chi_i\}$ subordinate to the cover $\{\wt V_i\}$ such that the differentials $\big|\rd \chi_i\big|$ are all bounded.
Hence we may use a Mayer-Vietoris argument to compute $H_2^*(\wt V)_\ell$ from the $H_2^*(\wt V_i)_\ell$: Since any intersection $\wt U$ of sets $\wt V_i$, $i > 0$, is $G_\lambda$-stable for some $\lambda = \lambda_i$, we have $H_2^*(\wt U)_\ell = 0$ and consequently obtain $H_2^*(\wt V)_\ell = 0$ by iteration of the standard Mayer-Vietoris argument.

Now let $0 \neq \alpha \in \cH^*_{k,\ell}$. As $\big(\unicov{k},g\big)$ is complete, $\alpha$ defines a non-zero element in $H_2^*(\unicov{k})_\ell$. But then, $\alpha|_{\wt V} = \rd \gamma$ for some smooth and square-integrable form $\gamma$ on $\wt V$, since $H_2^*(\wt V)_\ell = 0$. Using any smooth extension $\tilde \gamma$ of $\gamma$ to $\unicov{k}$, we see that $\beta = \alpha - \rd \tilde\gamma$ is compactly supported and closed, thus defines an element in $H_c^*(\unicov{k})_\ell$. Its projection onto $\cH^*_{k,\ell}$ is precisely $\alpha$.
\end{proof}
With these preparations at hand, we can address the proof of Theorem \ref{sen.1}.
\begin{proof}[Proof of Thm. \ref{sen.1}]
The obvious map $H_c^*(\unicov{k}) \longrightarrow H^*(\unicov{k})$, induced by inclusion $\Omega_c^* \hookrightarrow \Omega^*$, factors through $\cH^*_k$ as is shown in \cite[1.4]{SS} for instance. Thus we have maps
 \[ H_c^*(\unicov{k})_\ell \xrightarrow{\;\;\sigma\;\;} \cH^*_{k,\ell} \xrightarrow{\;\;\tau\;\;} H^*(\unicov{k})_\ell \;. \]
We have just shown $\sigma$ to be surjective, and $\tau$ is injective since $\big(\unicov{k},g\big)$ is complete and because $H_2^*(\unicov{k})_\ell$ is a subcomplex of $H^*(\unicov{k})_\ell$. This is to say that $\cH^*_k$ is canonically isomorphic to $\Im (\tau \circ \sigma)$, i.e., to \eqref{sen.image}. Using results of \cite{SS}, this implies the coprime case of the Sen Conjecture, (S.1): There, it is shown that \eqref{eq:sen.2} holds for $H^*(\unicov{k})_\ell$ and $H_c^*(\unicov{k})_\ell$ and that Poincar\'e-duality gives $H_c^{2k-2}(\unicov{k})_\ell \cong H^{2k-2}(\unicov{k})_\ell$. But then,
 \[ \cH_{k,\ell}^i \cong \Im \Big( H_c^i(\unicov{k})_\ell \longrightarrow H^i(\unicov{k})_\ell\Big)
     \cong \begin{cases} \CC &\text{if $i = 2k-2$,} \\ 0 &\text{else.} \end{cases} \qedhere \]
\end{proof}

\newpage

\bibliographystyle{acm}
\bibliography{mcasc_fin}

\end{document}